\documentclass[preprint,12pt]{elsarticle}
\usepackage[T1]{fontenc}
\usepackage[utf8]{inputenc}
\usepackage{mathrsfs,amsmath,amssymb,amsthm,mathtools,bbm}
\usepackage{enumitem}
\usepackage{xcolor,graphicx}
\usepackage[a4paper,margin=1in]{geometry}
\definecolor{Bittersweet}{HTML}{C04F17}
\usepackage[unicode,bookmarksnumbered,colorlinks,plainpages,linktocpage]{hyperref}
\definecolor{Darkgreen}{rgb}{0,0.4,0}
\hypersetup{%
    pdfborder = {0 0 0},
    colorlinks,
    citecolor=blue,
    filecolor=Darkgreen,
    linkcolor=Bittersweet,
    urlcolor=blue}

\usepackage{comment}

\newtheorem{theorem}{Theorem}
\newtheorem{proposition}[theorem]{Proposition}
\newtheorem{corollary}[theorem]{Corollary}
\newtheorem{lemma}[theorem]{Lemma}

\theoremstyle{definition}
\newtheorem{definition}[theorem]{Definition}
\newtheorem{remark}[theorem]{Remark}
\newtheorem{example}[theorem]{Example}

\newcommand{\defop}{\operatorname{def}}
\newcommand{\Pred}{\operatorname{Pred}}

\begin{document}
\hyphenpenalty=10000
\exhyphenpenalty=10000
\sloppy
%%%%%%%%%%%%%%%%%%%%%%%%%%%%%%%%%%%%%%%%%%%%%%%%%
%%%%%%%%%%%%%%%%%%%%%%%%%%%%%%%%%%%%%%%%%%%%%%%%%
\begin{frontmatter}
%%%%%%%%%%%%%%%%%%%%%%%%%%%%%%%%%%%%%%%%%%%%%%%%%
\title{Symmetric Grassmann Formulas: Monotone Dimension-Defect Hierarchies}
%%%%%%%%%%%%%%%%%%%%%%%%%%%%%%%%%%%%%%%%%%%%%%%%%
\author{Masoud Gharahi}
\ead{masoud.gharahi@gmail.com; masoud.gharahighahi@uj.edu.pl}
\author{Diego Ponterio}
\address{Faculty of Physics, Astronomy and Applied Computer Science, \\Jagiellonian University, 30-348 Kraków, Poland}
%%%%%%%%%%%%%%%%%%%%%%%%%%%%%%%%%%%%%%%%%%%%%%%%%
\begin{abstract}
We study the dimension defect of finitely many subspaces over an arbitrary field. We derive a symmetric, nonrecursive Grassmann-type formula for the dimension of their sum. The formula expresses the total dimension loss through intersections of a distinguished subspace with partial sums of the remaining subspaces, with coefficients determined by the number of subspaces involved. We also show that the formula admits a Shapley-value interpretation for the associated representable polymatroid. Grouping the correction terms by the number of participating subspaces yields a nonnegative dimension-defect profile. We prove that this profile is monotone and identify its successive gaps with the discrete curvatures of the average-rank profile. These gaps give exact remainders in two-sided defect bounds. Equality in either bound holds precisely when the images of the subspaces in the quotient by their common intersection form an internal direct sum. We also give an exact geometric decomposition of the Kinser slack into nonnegative quotient dimensions and characterize equality. Averaging these slacks over permutations and contractions recovers every defect curvature except the final one; additional representability constraints remain in the individual ordered slacks. Weighted and dual formulas accompany the expansion, and entropy analogues express the defect levels and curvatures as averages of mutual and conditional mutual information, respectively.
\end{abstract}
%%%%%%%%%%%%%%%%%%%%%%%%%%%%%%%%%%%%%%%%%%%%%%%%%
%%%%%%%%%%%%%%%%%%%%%%%%%%%%%%%%%%%%%%%%%%%%%%%%%
\end{frontmatter}
%%%%%%%%%%%%%%%%%%%%%%%%%%%%%%%%%%%%%%%%%%%%%%%%%

%%%%%%%%%%%%%%%%%%%%%%%%%%%%%%%%%%%%%%%%%%%%%%%%%
\section{Introduction}
%%%%%%%%%%%%%%%%%%%%%%%%%%%%%%%%%%%%%%%%%%%%%%%%%
Let $V$ be a finite-dimensional vector space over an arbitrary field $\mathbbm{F}$, and let $V_1,\ldots,V_n$ be subspaces of $V$, with $n\geq2$. We write $[n]\coloneq\{1,\ldots,n\}$ and denote by $\mathfrak{S}_n$ the symmetric group on $[n]$.

For two subspaces $V_1,V_2\subseteq V$, the classical Grassmann formula \cite{Grassmann} reads
\begin{equation}\label{Grassmann-classical}
\dim(V_1+V_2)=\dim(V_1)+\dim(V_2)-\dim(V_1\cap V_2).
\end{equation}
It measures exactly the loss of dimension caused by overlap. For three or more subspaces, the direct inclusion--exclusion analogue obtained by alternating dimensions of ordinary intersections fails \cite{Axler24}. The reason is that the lattice of subspaces is modular but generally not distributive \cite{Birkhoff67}. For subspaces $U,V,W$, the intersection $U\cap(V+W)$ need not equal $(U\cap V)+(U\cap W)$. Thus a subspace can intersect $V$ and $W$ trivially while intersecting their sum nontrivially, an overlap that is invisible to the ordinary inclusion--exclusion.

A homological approach to this failure is developed by Day and Wade \cite[Sections~3.1--3.4]{DW18}, who associate a chain complex to a subspace arrangement using ordinary intersections. The additivity of dimension along short exact sequences is encoded by the Grothendieck group of finite-dimensional vector spaces. Applied to this complex, the Euler-characteristic identity expresses the dimension of the sum as ordinary inclusion--exclusion together with an alternating correction from the homology dimensions \cite[Remark~3.10]{DW18} \footnote{Explicitly, set $C_0=\sum_{i\in[n]}V_i$, and $C_k=\bigoplus_{\substack{S\subseteq[n]\\|S|=k}} \bigcap_{j\in S}V_j$ for $1\leq k\leq n$. Also set $C_{n+1}=\{0\}$. The boundary maps $\partial_k:C_k\to C_{k-1}$ are alternating inclusions for $2\leq k\leq n$, while $\partial_1$ is addition and
$\partial_{n+1}:C_{n+1}\to C_n$ is the zero map. The homology spaces $\mathcal{H}_k=\ker(\partial_k)/\operatorname{im}(\partial_{k+1})$ measure the failure of exactness. Applying the standard Euler-characteristic identity to the complex of Day and Wade~\cite{DW18} gives:
\[
\dim\Bigl(\sum_{i\in[n]}V_i\Bigr)=\sum_{k\in[n]}(-1)^{k+1}
\sum_{\substack{\emptyset\neq S\subseteq[n]\\|S|=k}}
\dim\Bigl(\bigcap_{j\in S}V_j\Bigr)+\sum_{k\in[n]}(-1)^k\dim(\mathcal{H}_k).
\]
}.

An exact identity for subspace sums also addresses this failure. Tian \cite[Theorem~2]{Tian19} derived a sum formula over $\mathbbm{C}$ using matrix and generalized-inverse methods. Makrides and Szemberg \cite[Theorem~1.1]{MS25} gave an elementary proof, valid over an arbitrary field, of the symmetric identity
\begin{equation}\label{TMS-intro}
(n-1)\dim\Bigl(\sum_{i\in[n]}V_i\Bigr)=\sum_{i\in[n]}
\dim\Bigl(\sum_{j\in[n]\setminus\{i\}}V_j\Bigr)-
\dim\Bigl(\bigcap_{i\in[n]}\sum_{j\in[n]\setminus\{i\}}V_j\Bigr),
\end{equation}
which appeared earlier in \cite{Tian19}. Tian \cite[Theorem~3.1]{Tian25} also established recursive sum and intersection formulas and identities involving mixed operations.

Our first aim is to obtain a symmetric, nonrecursive expansion in which every correction term is specified directly by the original subspaces. We achieve this by averaging the ordered form of the two-subspace Grassmann formula. If the subspaces indexed by $S\subseteq[n]\setminus\{i\}$ precede $V_i$ in an ordering, then the dimension contributed by $V_i$ at that step is
\begin{equation}\label{marginal-intro}
\dim(V_i)-\dim\Bigl(V_i\cap\sum_{j\in S}V_j\Bigr).
\end{equation}
Averaging over all orderings gives the symmetric nonrecursive formula
\begin{equation}\label{Sym-Grassmann-intro}
\dim\Bigl(\sum_{i\in[n]}V_i\Bigr)=\sum_{i\in[n]}\dim(V_i)-
\sum_{i\in[n]}\sum_{\emptyset\neq S\subseteq[n]\setminus\{i\}}\frac{\dim\bigl(V_i\cap\sum_{j\in S}V_j\bigr)}{(|S|+1)\binom{n}{|S|+1}}.
\end{equation}
The coefficients and intersection terms are specified directly from the original family, with no recursive application of a formula for fewer subspaces. Each correction records the part of a distinguished subspace already contained in a partial sum, and its coefficient depends only on the number of subspaces in that sum. The formula is invariant under relabeling and reduces to the classical Grassmann formula for $n=2$. An alternative inductive proof is given in~\ref{app:inductive-proof}.

The same averaging procedure gives a weighted extension. An arbitrary probability distribution on orderings yields a generally nonsymmetric formula, while the uniform distribution gives \eqref{Sym-Grassmann-intro}. Passing to annihilators gives a dual formula for the dimension of an intersection. We also compare \eqref{Sym-Grassmann-intro} directly with \eqref{TMS-intro}, showing that the two expressions resolve the same total correction through different families of subspace-intersection data.

A conceptual contribution is the natural interpretation of the symmetric Grassmann-type formula in polymatroid theory and cooperative game theory. The set function r defined, for $S\subseteq[n]$, by
\begin{equation}\label{rank-function-intro}
\mathfrak{r}(S)\coloneq\dim\Bigl(\sum_{i\in S}V_i\Bigr),
\qquad S\subseteq[n],
\end{equation}
with $\mathfrak{r}(\emptyset)=0$, is the rank function of a linearly representable integral polymatroid; in particular, it is normalized, monotone, and submodular \cite{Fujishige05}. The identity
\[
\frac{|S|!(n-|S|-1)!}{n!}=\frac{1}{(|S|+1)\binom{n}{|S|+1}}
\]
shows that the coefficients in \eqref{Sym-Grassmann-intro} are precisely the Shapley coefficients \cite{Shapley53}. Accordingly, the symmetric Grassmann-type formula is the Shapley efficiency identity for the rank function $\mathfrak{r}$, written in terms of subspace intersections; this interpretation is developed in Section~\ref{sec:shapley}.

We next use the correction terms in \eqref{Sym-Grassmann-intro} to construct a dimension-defect hierarchy that resolves the total dimension defect according to the number of participating subspaces. The resulting profile records how average overlap changes as more subspaces are combined. Define the dimension defect by
\begin{equation}\label{defect-intro}
\defop(V_1,\ldots,V_n)\coloneq\sum_{i\in[n]}\dim(V_i)
-\dim\Bigl(\sum_{i\in[n]}V_i\Bigr).
\end{equation}
Grouping the correction terms by the total number of participating subspaces gives the canonical decomposition
\begin{equation}\label{defect-sum-intro}
\defop(V_1,\ldots,V_n)=\sum_{k=2}^{n}A_k,
\end{equation}
where $A_k$ is the average overlap between one distinguished subspace and the sum of $k-1$ other subspaces, taken over all such choices. We prove monotonicity of the dimension-defect profile. Specifically, the average overlap sequence $(A_k)_{k=2}^n$ is nondecreasing with interaction order $k$:
\begin{equation}\label{monotone-intro}
A_2\leq A_3\le\cdots\le A_n.
\end{equation}
Here $A_2$ is the average pairwise intersection dimension, while $A_n$ is the average dimension of the intersection of a distinguished subspace with the sum of the remaining subspaces. The profile therefore detects dependence that is invisible to pairwise intersections alone. Its levels are indexed by the number of participating subspaces, but they do not isolate interactions involving exactly that many subspaces, since overlap among fewer subspaces may also contribute at higher levels.

The successive gaps $B_k\coloneq A_{k+1}-A_k$, for $2\leq k\leq n-1$, give a finer description of this hierarchy. Let $R_k$ denote the average dimension of the sum of a uniformly chosen collection of $k$ subspaces, with $R_0=0$, and set $\Delta_k\coloneq R_k-R_{k-1}$ for $1\leq k\leq n$. We prove that
\begin{equation}
A_k=R_1-\Delta_k \quad (2\leq k\leq n), \qquad
B_k=\Delta_k-\Delta_{k+1} \quad (2\leq k\leq n-1).
\end{equation}
Thus $\Delta_k$ is the average marginal dimension contributed when the number of selected subspaces increases from $k-1$ to $k$, while $A_k$ measures its loss relative to the first step and $B_k$ its decrease at the next step. In particular, the average-rank profile $(R_k)_{k=0}^{n}$ is discretely concave, and the gaps $B_k=2R_k-R_{k-1}-R_{k+1}$, for $2\leq k\leq n-1$, measure its discrete curvature.

The defect decomposition and monotonicity yield the endpoint bounds
\begin{equation}
(n-1)A_2\leq\defop(V_1,\ldots,V_n)\leq(n-1)A_n,
\end{equation}
and the successive gaps give exact remainder formulas for both estimates. A further principal aim is to characterize the subspace arrangements attaining these bounds. We prove that equality in either bound holds precisely when, with $D\coloneq\bigcap_{i\in[n]}V_i$, the quotient subspaces $V_i/D$ form an internal direct sum in $V/D$ (Theorem~\ref{thm:common-core}). This common-core characterization identifies the arrangements whose dependence is entirely accounted for by their shared intersection. Annihilators yield a parallel hierarchy for intersections. Section~\ref{sec:defect} develops these results.

The graded defect profile also connects with representability-specific linear-rank inequalities. Ingleton's inequality \cite{Ingleton71} and Kinser's inequalities \cite{Kinser11} constrain linearly representable polymatroids beyond submodularity. A further principal contribution is an exact geometric decomposition of the Kinser slack. Retaining the nonnegative remainders in Kinser's proof expresses this slack as a sum of explicit quotient dimensions and yields necessary and sufficient geometric conditions for equality (Theorem~\ref{thm:Kinser-slack-decomposition} and Corollary~\ref{cor:Kinser-equality}). This decomposition is related to Friedman and Tootooni's study of failures of distributivity through discoordination \cite[Corollary~3.10(5)]{FT23}.

We connect this geometric decomposition with the defect curvatures by showing that the Kinser slack, averaged over all orderings, equals $(n-2)B_2$. More generally, quotienting by the sum of $k-2$ selected subspaces and averaging over all such selections and all orderings of the remaining subspaces gives $(n-k)B_k$ for $2\leq k\leq n-2$. Consequently, for these values of $k$, $B_k=0$ if and only if every such ordered quotient arrangement saturates the corresponding Kinser inequality. The final gap $B_{n-1}$ has no counterpart in this contraction hierarchy. These averages recover the submodular curvature and yield no additional symmetric inequality; the representability-specific information remains in the individual ordered slacks. Section~\ref{subsec:Kinser} develops this connection.

The average-rank identities have an information-theoretic counterpart. For arbitrary finite-valued random variables, the entropy analogues of $A_k$ and $B_k$ are averages of mutual information and conditional mutual information, respectively. Summing the entropy analogues of $A_k$ over $k$ gives total correlation, and the average entropy profile is discretely concave. For a linear source over a finite field $\mathbbm{F}_q$ with $q$ elements, the entropy analogues are exactly $A_k\log q$ and $B_k\log q$. These correspondences are stated and proved in~\ref{app:information-theoretic}.

The paper is organized as follows. Section~\ref{sec:G-formulas} establishes the symmetric formula, its weighted and dual extensions, and its comparison with the Tian--Makrides--Szemberg formula. Section~\ref{sec:shapley} develops the polymatroid and Shapley values interpretations. Section~\ref{sec:defect} develops the defect hierarchy, the common-core equality characterization, and the average-rank, dual, and Kinser connections. The final section gives a discussion and outlook.

%%%%%%%%%%%%%%%%%%%%%%%%%%%%%%%%%%%%%%%%%%%%%%%%%
\section{Grassmann-type formulas}\label{sec:G-formulas}
%%%%%%%%%%%%%%%%%%%%%%%%%%%%%%%%%%%%%%%%%%%%%%%%%
This section establishes the Grassmann-type formulas underlying the dimension-defect hierarchy. We first derive the symmetric formula by averaging the ordered Grassmann identity, then extend the averaging procedure to arbitrary distributions on orderings. We next obtain a dual formula for intersections and finally compare the symmetric formula with the Tian--Makrides--Szemberg formula \cite{Tian19,MS25}.
%%%%%%%%%%%%%%%%%%%%%%%%%%%%%%%%%%%%%%%%%%%%%%%%%
\subsection{Symmetric formula}\label{subsec:sym-Grassmann}
%%%%%%%%%%%%%%%%%%%%%%%%%%%%%%%%%%%%%%%%%%%%%%%%%
We begin with the standard ordered iteration of the two-subspace Grassmann formula \cite{Axler24}, which will be averaged to obtain the symmetric expansion.

\begin{lemma}[Ordered Grassmann expansion]\label{lem:ordered}
Let $V_1,\ldots,V_n\subseteq V$, and let $\pi\in\mathfrak{S}_n$ be a permutation. Then
\begin{align}\label{ordered-Grassmann}
\dim\Bigl(\sum_{i\in[n]}V_i\Bigr)=\sum_{k\in[n]}
\Bigl[\dim(V_{\pi(k)})-\dim\Bigl(V_{\pi(k)}\cap\sum_{l<k}V_{\pi(l)}\Bigr)\Bigr],
\end{align}
where the sum over $l<1$ is understood to be $\{0\}$.
\end{lemma}
\begin{proof}
Set $U_k\coloneq\sum_{l=1}^{k}V_{\pi(l)}$ and $U_0\coloneq\{0\}$. By the classical Grassmann formula,
\begin{equation*}
\dim(U_k)-\dim(U_{k-1})=\dim(V_{\pi(k)})-\dim(U_{k-1}\cap V_{\pi(k)}).
\end{equation*}
Summing over $k\in[n]$ telescopes to \eqref{ordered-Grassmann}.
\end{proof}

For $i\in[n]$ and $\pi\in\mathfrak{S}_n$, let
\begin{equation}\label{precede}
\Pred_{\pi}(i)\coloneq
\{j\in[n]\setminus\{i\}\,\mid\,\pi^{-1}(j)<\pi^{-1}(i)\}
\end{equation}
be the set of indices that precede $i$ in the ordering $\pi$.

For $1\leq k\leq n$, set
\begin{equation}\label{cnk}
c_{n,k}\coloneq\frac{1}{k\binom{n}{k}}.
\end{equation}

\begin{theorem}[Symmetric Grassmann-type formula]\label{thm:symmetric-Grassmann}
Let $V_1,\ldots,V_n\subseteq V$. Then
\begin{equation}\label{sym-Grassmann}
\dim\Bigl(\sum_{i\in[n]}V_i\Bigr)=\sum_{i\in[n]}
\Bigl[\dim(V_i)-\sum_{\emptyset\neq S\subseteq[n]\setminus\{i\}}c_{n,|S|+1}\,\dim\Bigl(V_i\cap\sum_{j\in S}V_j\Bigr)\Bigr].
\end{equation}
\end{theorem}
\begin{proof}
Average \eqref{ordered-Grassmann} uniformly over all $n!$ permutations. Fix $i\in[n]$ and $S\subseteq[n]\setminus\{i\}$ with $|S|=s$. The condition $\Pred_{\pi}(i)=S$ means that precisely the elements of $S$ occur before $i$; they can be ordered in $s!$ ways, while the remaining $n-s-1$ elements can be ordered after $i$ in $(n-s-1)!$ ways. Thus, for a uniformly random permutation $\pi$, the probability that exactly the elements of $S$ precede $i$ is
\begin{equation*}
\mathtt{p}\bigl(\Pred_{\pi}(i)=S\bigr)=\frac{s!(n-s-1)!}{n!}=\frac{1}{n\binom{n-1}{s}}=\frac{1}{(s+1)\binom{n}{s+1}}=c_{n,s+1}.
\end{equation*}
Therefore the average contribution of the intersection loss associated with $(i,S)$ is $c_{n,|S|+1}\dim(V_i\cap\sum_{j\in S}V_j)$. Summing over $i$ and all nonempty $S$ gives \eqref{sym-Grassmann}.
\end{proof}

An alternative proof by induction on the number of subspaces is given in~\ref{app:inductive-proof}.

\begin{remark}\label{rem:n2}
For $n=2$, Theorem~\ref{thm:symmetric-Grassmann} reduces exactly to the classical Grassmann formula. Indeed, $c_{2,2}=1/2$, and the same intersection $V_1\cap V_2$ occurs once with distinguished index $1$ and once with distinguished index $2$.
\end{remark}

\begin{example}
For $n=3$, Theorem~\ref{thm:symmetric-Grassmann}
recovers the identity in \cite[Exercise~20, Section~2C]{Axler24}:
\begin{equation}\label{sum-n=3}
\dim(V_1+V_2+V_3)=\sum_{i=1}^{3}\dim(V_i)
-\frac{1}{3}\sum_{1\leq i<j\le3}\dim(V_i\cap V_j)-\frac{1}{3}\sum_{i=1}^{3}
\dim\Bigl(V_i\cap\!\sum_{j\in[3]\setminus\{i\}}V_j\Bigr).
\end{equation}
Here the coefficient $1/3=2\,c_{3,2}$ arises by collecting
the identical terms corresponding to $V_i\cap V_j$ and
$V_j\cap V_i$. More generally, collecting the pairwise
terms over $i<j$ gives the coefficient
$2\,c_{n,2}=1/\binom{n}{2}$.
The first correction records pairwise overlap, while the
second captures, in particular, the genuinely non-distributive
phenomenon that a subspace may meet the sum of two others
nontrivially even when all pairwise intersections vanish.
\end{example}

\begin{example}\label{ex:three-lines}
Let $V=\mathbbm{F}^2$, with standard basis $\{e_1,e_2\}$, and consider the three distinct lines through the origin
\begin{equation*}
V_1=\operatorname{span}\{e_1\}, \qquad
V_2=\operatorname{span}\{e_2\}, \qquad
V_3=\operatorname{span}\{e_1+e_2\}.
\end{equation*}
Then $V_i\cap V_j=\{0\}$ for $i\neq j$, whereas $V_i\cap(V_j+V_k)=V_i$ for each permutation $(i,j,k)$ of $(1,2,3)$. Formula \eqref{sum-n=3} therefore gives
\begin{equation*}
\dim(V_1+V_2+V_3)=3-\frac{1}{3}(1+1+1)=2.
\end{equation*}
Thus the terms involving intersections with partial sums recover exactly the information missed by ordinary inclusion--exclusion.
\end{example}

%%%%%%%%%%%%%%%%%%%%%%%%%%%%%%%%%%%%%%%%%%%%%%%%%
\subsection{Weighted orderings}\label{subsec:weighted}
%%%%%%%%%%%%%%%%%%%%%%%%%%%%%%%%%%%%%%%%%%%%%%%%%
The uniform average is not necessary for the underlying identity. Any probability distribution on orderings produces a valid, generally nonsymmetric, Grassmann-type expansion. This extension is the subspace form of the classical random-order construction \cite{Weber88}.

Let $\mu$ be a probability distribution on the orderings $\mathfrak{S}_n$. For $i\in[n]$ and $S\subseteq[n]\setminus\{i\}$, let $\mathtt{p}_{\mu}(i,S)$ be the probability that exactly the elements of $S$ occur before $i$ in an ordering chosen according to $\mu$; explicitly,
\begin{equation}\label{p-mu}
\mathtt{p}_{\mu}(i,S)\coloneq\mu\bigl(\{\pi\in\mathfrak{S}_n\,\mid\,\Pred_{\pi}(i)=S\}\bigr).
\end{equation}

\begin{theorem}[Weighted Grassmann-type formula]\label{thm:weighted}
Let $V_1,\ldots,V_n\subseteq V$. Then
\begin{align}\label{weighted-Grassmann}
\dim\Bigl(\sum_{i\in[n]}V_i\Bigr)=\sum_{i\in[n]}\dim(V_i)
-\sum_{i\in[n]}\sum_{S\subseteq[n]\setminus\{i\}}\mathtt{p}_{\mu}(i,S)
\dim\Bigl(V_i\cap\sum_{j\in S}V_j\Bigr).
\end{align}
The term $S=\emptyset$ vanishes because the empty sum of subspaces is $\{0\}$.
\end{theorem}
\begin{proof}
For any fixed ordering $\pi$, Lemma~\ref{lem:ordered} adds the dimensions of the $V_i$ and subtracts, for each $i$, its intersection with the sum of the subspaces preceding it. Thus \eqref{ordered-Grassmann} can be written as
\begin{equation*}
\dim\Bigl(\sum_{i\in[n]}V_i\Bigr)=\sum_{i\in[n]}\dim(V_i)
-\sum_{i\in[n]}\dim\Bigl(V_i\cap\sum_{j\in\Pred_{\pi}(i)}V_j\Bigr).
\end{equation*}
The left-hand side and the first sum do not depend on $\pi$. Taking the expectation over $\pi$ with distribution $\mu$, and grouping the orderings by the set $S=\Pred_{\pi}(i)$, gives
\begin{equation*}
\mathtt{E}_{\mu}\Bigl[\dim\Bigl(V_i\cap\sum_{j\in\Pred_{\pi}(i)}V_j\Bigr)\Bigr]=
\sum_{S\subseteq[n]\setminus\{i\}}\mathtt{p}_{\mu}(i,S)
\dim\Bigl(V_i\cap\sum_{j\in S}V_j\Bigr).
\end{equation*}
Substituting this identity for each $i$ proves
\eqref{weighted-Grassmann}.
\end{proof}

\begin{example}\label{ex:weighted}
Let $V=\mathbbm{F}^2$, $V_1=\operatorname{span}\{e_1\}$, and $V_2=V$. Give the ordering $(1,2)$ probability $3/4$ and the ordering $(2,1)$ probability $1/4$. In the first ordering, adding $V_1$ and then $V_2$ contributes $1+(2-1)=2$; in the second, adding $V_2$ and then $V_1$ contributes $2+(1-1)=2$. Their weighted average is
\begin{equation*}
\frac{3}{4}[1+(2-1)]+\frac{1}{4}[2+(1-1)]=2.
\end{equation*}
In the notation of the theorem, $\mathtt{p}_{\mu}(2,\{1\})=3/4$ and $\mathtt{p}_{\mu}(1,\{2\})=1/4$. Since $V_1\cap V_2=V_1$ has dimension $1$, formula \eqref{weighted-Grassmann} gives the same result:
\begin{equation*}
\dim(V_1+V_2)=1+2-\frac{3}{4}(1)-\frac{1}{4}(1)=2.
\end{equation*}
The weights change the average contributions of $V_1$ and $V_2$ to $3/4$ and $5/4$, respectively, while their sum remains $2$.
\end{example}

\begin{remark}\label{rem:canonical}
Theorem~\ref{thm:weighted} exhibits a family of exact Grassmann-type formulas parametrized by distributions on orderings. The uniform distribution is the canonical label-symmetric choice. In this case $\mathtt{p}_{\mu}(i,S)=c_{n,|S|+1}$, recovering Theorem~\ref{thm:symmetric-Grassmann}.
\end{remark}

%%%%%%%%%%%%%%%%%%%%%%%%%%%%%%%%%%%%%%%%%%%%%%%%%
\subsection{A dual formula for intersections}\label{subsec:dual-formula}
%%%%%%%%%%%%%%%%%%%%%%%%%%%%%%%%%%%%%%%%%%%%%%%%%
The symmetric Grassmann-type formula \eqref{sym-Grassmann} has a dual counterpart obtained by the standard annihilator construction \cite{Axler24}. Let $V^\vee\coloneq\operatorname{Hom}_{\mathbbm{F}}(V,\mathbbm{F})$ be the algebraic dual of $V$. For a subspace $U\subseteq V$, write $U^{\circ}\coloneq\{f\in V^\vee\,\mid\,f(u)=0,~\forall\,u\in U\}$. For every nonempty $S\subseteq[n]$, one has
\begin{equation}\label{annihilator}
\Bigl(\sum_{i\in S}V_i\Bigr)^{\circ}=\bigcap_{i\in S}V_i^{\circ}, \qquad
\Bigl(\bigcap_{i\in S}V_i\Bigr)^{\circ}=\sum_{i\in S}V_i^{\circ},
\end{equation}
and $\dim(U^{\circ})=\dim(V)-\dim(U)$.

\begin{theorem}[Dual symmetric Grassmann-type formula]\label{thm:dual-Grassmann}
Let $V_1,\ldots,V_n\subseteq V$. Then
\begin{equation}\label{dual-Grassmann}
\dim\Bigl(\bigcap_{i\in[n]}V_i\Bigr)=\sum_{i\in[n]}\dim(V_i)-\sum_{i\in[n]}\sum_{\emptyset\neq S\subseteq[n]\setminus\{i\}}c_{n,|S|+1}\,\dim\Bigl(V_i+\bigcap_{j\in S}V_j\Bigr).
\end{equation}
\end{theorem}
\begin{proof}
Apply Theorem~\ref{thm:symmetric-Grassmann} in $V^\vee$ to $V_1^{\circ},\ldots,V_n^{\circ}$. Let $d=\dim(V)$. Using \eqref{annihilator},
\begin{equation*}
\dim\Bigl(V_i^{\circ}\cap\sum_{j\in S}V_j^{\circ}\Bigr)=
d-\dim\Bigl(V_i+\bigcap_{j\in S}V_j\Bigr).
\end{equation*}
Moreover,
\begin{equation*}
\dim\Bigl(\sum_{i\in[n]}V_i^{\circ}\Bigr)=d-\dim\Bigl(\bigcap_{i\in[n]}V_i\Bigr),
\qquad \dim(V_i^{\circ})=d-\dim(V_i).
\end{equation*}
It remains only to observe that the sum of all coefficients in the double correction term equals $n-1$. Indeed,
\begin{align*}
\sum_{i\in[n]}
\sum_{\emptyset\neq S\subseteq[n]\setminus\{i\}}c_{n,|S|+1}=\sum_{i\in[n]}
\sum_{\emptyset\neq S\subseteq[n]\setminus\{i\}}
\frac{1}{(|S|+1)\binom{n}{|S|+1}}&=\sum_{i\in[n]}\sum_{s=1}^{n-1}
\binom{n-1}{s}\frac{1}{(s+1)\binom{n}{s+1}}\\
&=\sum_{i\in[n]}\sum_{s=1}^{n-1}\frac{1}{n}=n-1.
\end{align*}
The ambient-dimension terms therefore cancel, leaving \eqref{dual-Grassmann}.
\end{proof}

\begin{remark}
For $n=2$, formula \eqref{dual-Grassmann} gives
\begin{equation*}
\dim(V_1\cap V_2)=\dim(V_1)+\dim(V_2)-\dim(V_1+V_2),
\end{equation*}
which is again the classical Grassmann identity.
\end{remark}

%%%%%%%%%%%%%%%%%%%%%%%%%%%%%%%%%%%%%%%%%%%%%%%%%
\subsection{Comparison with the Tian--Makrides--Szemberg formula}\label{subsec:TMS-comparison}
%%%%%%%%%%%%%%%%%%%%%%%%%%%%%%%%%%%%%%%%%%%%%%%%%
We next compare Theorem~\ref{thm:symmetric-Grassmann} directly with the formula of Tian \cite{Tian19} and Makrides--Szemberg \cite{MS25}.

For each $i\in[n]$, set
\begin{equation}\label{Wi-Y}
W_i\coloneq\sum_{j\in[n]\setminus\{i\}}V_j,
\qquad Y\coloneq\bigcap_{i\in[n]}W_i.
\end{equation}
In this notation, the Tian--Makrides--Szemberg identity recalled in \eqref{TMS-intro} takes the form
\begin{equation}\label{TMS-formula}
(n-1)\dim\Bigl(\sum_{i\in[n]}V_i\Bigr)=
\sum_{i\in[n]}\dim(W_i)-\dim(Y).
\end{equation}
Makrides and Szemberg \cite[Corollary~2.2]{MS25} also record the equivalent identity
\begin{equation}\label{MS-corollary}
\dim\Bigl(\sum_{i\in[n]}V_i\Bigr)+\sum_{i\in[n]}\dim(V_i\cap W_i)=\sum_{i\in[n]}\dim(V_i)+\dim(Y).
\end{equation}
Formula~\eqref{MS-corollary} provides a convenient way to compare their global correction with our partial-sum corrections. The following identity is the direct algebraic comparison of these two expressions for the same defect.

\begin{theorem}\label{thm:TMS-comparison}
Let $V_1,\ldots,V_n\subseteq V$, and let $W_i$ and $Y$ be defined by \eqref{Wi-Y}. Then
\begin{equation}\label{TMS-comparison}
\sum_{i\in[n]}\sum_{\emptyset\neq S\subseteq[n]\setminus\{i\}}c_{n,|S|+1}\,\dim\Bigl(V_i\cap\sum_{j\in S}V_j\Bigr)=\sum_{i\in[n]}\dim(V_i\cap W_i)-\dim(Y).
\end{equation}
\end{theorem}
\begin{proof}
Rearranging Theorem~\ref{thm:symmetric-Grassmann} shows that the left-hand side of \eqref{TMS-comparison} is
\begin{equation*}
\sum_{i\in[n]}\dim(V_i)-\dim\Bigl(\sum_{i\in[n]}V_i\Bigr).
\end{equation*}
On the other hand, rearranging \eqref{MS-corollary} gives
\begin{equation*}
\sum_{i\in[n]}\dim(V_i)-\dim\Bigl(\sum_{i\in[n]}V_i\Bigr)=
\sum_{i\in[n]}\dim(V_i\cap W_i)-\dim(Y),
\end{equation*}
which proves the claim.
\end{proof}

\begin{remark}\label{rem:TMS-structural}
Theorem~\ref{thm:TMS-comparison} makes explicit that the two identities \eqref{sym-Grassmann} and \eqref{TMS-formula} resolve the same total dimension correction through different subspace data. The Tian--Makrides--Szemberg expression uses the maximal proper sums $W_i$ and their common intersection $Y$. The symmetric Grassmann-type identity instead uses intersections of a distinguished subspace with partial sums indexed by all nonempty subsets of the remaining indices. These terms can be grouped by cardinality to form the defect profile studied below.
\end{remark}

%%%%%%%%%%%%%%%%%%%%%%%%%%%%%%%%%%%%%%%%%%%%%%%%%
\section{Representable polymatroids and Shapley values}\label{sec:shapley}
%%%%%%%%%%%%%%%%%%%%%%%%%%%%%%%%%%%%%%%%%%%%%%%%%
We now make explicit the combinatorial structure behind Theorem~\ref{thm:symmetric-Grassmann} and interpret its terms as an allocation of the total dimension among the participating subspaces. The abstraction of linear dependence into matroidal structure goes back to Whitney \cite{Whitney35}. In the present setting, the ground elements are themselves subspaces, and the corresponding rank data are described by an integral polymatroid. For each $S\subseteq[n]$, define the rank function
\begin{equation}\label{rank-function-S}
\mathfrak{r}(S)\coloneq\dim\Bigl(\sum_{i\in S}V_i\Bigr),
\end{equation}
with $\mathfrak{r}(\emptyset)=0$.

The function $\mathfrak{r}:2^{[n]}\to\mathbbm{Z}_{\geq0}$ defined in \eqref{rank-function-S} is the rank function of a linearly representable integral polymatroid \cite{Fujishige05,Kinser11}. For completeness, we verify its defining properties. Normalization and monotonicity are immediate. For $S\subseteq[n]$, write $U_S\coloneq\sum_{i\in S}V_i$. For $A,B\subseteq[n]$, Grassmann's formula gives
\begin{equation}
\dim(U_A)+\dim(U_B)=\dim(U_A+U_B)+\dim(U_A\cap U_B).
\end{equation}
Since $U_A+U_B=U_{A\cup B}$ and $U_{A\cap B}\subseteq U_A\cap U_B$, it follows that
\begin{equation}
\mathfrak{r}(A)+\mathfrak{r}(B)\geq\mathfrak{r}(A\cup B)+\mathfrak{r}(A\cap B),
\end{equation}
which proves submodularity. Integrality follows from integer-valuedness of dimension, and linear representability is supplied by the subspaces $V_1,\ldots,V_n$ themselves.

Integral polymatroids admit several related combinatorial representations. Discrete polymatroids were developed systematically in \cite{HH02}, while the natural-matroid construction associated with an integer polymatroid was studied more recently in \cite{BCF23}. Here we retain the subspace representation itself, since this realization gives the Grassmann intersection terms appearing in Theorem~\ref{thm:symmetric-Grassmann}.

To interpret $\mathfrak{r}$ as a cooperative game, regard each index $i$ as a player represented by the subspace $V_i$, and $\mathfrak{r}(S)$ as the worth of the coalition $S$. The allocation question is how much of the total dimension $\mathfrak{r}([n])$ should be credited to each player. For this purpose, consider the new dimension contributed when player $i$ joins a coalition $S$ not containing it, i.e., $S\subseteq[n]\setminus\{i\}$. Applying the two-subspace Grassmann formula to $V_i$ and $\sum_{j\in S}V_j$ gives
\begin{equation}\label{rank-marginal}
\mathfrak{r}(S\cup\{i\})-\mathfrak{r}(S)
=\dim(V_i)-\dim\Bigl(V_i\cap\sum_{j\in S}V_j\Bigr).
\end{equation}
Thus the marginal contribution of $i$ is its subspace dimension minus the overlap with the sum already formed by the coalition. This contribution depends on which players have joined before $i$. The Shapley value assigns a share to each player by averaging its marginal contribution over all possible orderings.

More generally, for a cooperative game $v:2^{[n]}\to\mathbbm{R}$ with $v(\emptyset)=0$, the Shapley value of player $i$ is \cite{Shapley53}
\begin{equation}\label{Shapley-value}
\phi_i(v)=\sum_{S\subseteq[n]\setminus\{i\}}
\frac{|S|!(n-|S|-1)!}{n!}\bigl(v(S\cup\{i\})-v(S)\bigr).
\end{equation}
Here the coefficient is the probability that exactly the players in $S$ precede $i$ in a uniformly random ordering. The resulting shares satisfy the efficiency identity
\begin{equation}\label{Shapley-efficiency}
\sum_{i\in[n]}\phi_i(v)=v([n]),
\end{equation}
so the entire worth of the grand coalition is allocated.

The interpretation of cooperative-game values through random orderings and more general probability distributions was developed systematically in \cite{Weber88}. Shapley-type values in a matroidal setting, where the feasible coalitions are constrained by a matroid, were studied in \cite{BDJL01}. In the present setting, every subset of $[n]$ is an admissible coalition, and the subspaces determine its worth through $\mathfrak{r}$.

Combining the marginal identity \eqref{rank-marginal} with the Shapley weights expresses each player's share through averaged intersection losses. The following proposition identifies these weights with $c_{n,|S|+1}$ and shows that allocating the total dimension recovers the symmetric Grassmann formula.

\begin{proposition}[Shapley--Grassmann correspondence]
\label{prop:Shapley}
For the rank function $\mathfrak{r}$ in \eqref{rank-function-S},
\begin{equation}\label{Shapley-value-Grassmann}
\phi_i(\mathfrak{r})=\dim(V_i)-
\sum_{\emptyset\neq S\subseteq[n]\setminus\{i\}}
c_{n,|S|+1}\,\dim\Bigl(V_i\cap\sum_{j\in S}V_j\Bigr).
\end{equation}
Consequently, Theorem~\ref{thm:symmetric-Grassmann} is precisely the Shapley efficiency identity \eqref{Shapley-efficiency} for the linearly representable polymatroid rank function $\mathfrak{r}$, expressed through Grassmann intersection losses.
\end{proposition}
\begin{proof}
Substitute the Grassmann marginal identity \eqref{rank-marginal} into \eqref{Shapley-value} and use
\begin{equation*}
\frac{|S|!(n-|S|-1)!}{n!}=\frac{1}{(|S|+1)\binom{n}{|S|+1}}=c_{n,|S|+1}.
\end{equation*}
For each fixed $i$, these weights sum to one over $S\subseteq[n]\setminus\{i\}$, so the terms involving $\dim(V_i)$ combine to give $\dim(V_i)$. The term indexed by $S=\emptyset$ has zero intersection loss, yielding \eqref{Shapley-value-Grassmann}. Summing over $i$ and using \eqref{Shapley-efficiency} gives Theorem~\ref{thm:symmetric-Grassmann}.
\end{proof}

\begin{remark}[Expected marginal dimension]\label{rem:expected}
In the predecessor notation, the allocation to $V_i$ takes
the form
\begin{equation}
\phi_i(\mathfrak{r})=\mathtt{E}\Bigl[\dim(V_i)-
\dim\Bigl(V_i\cap\sum_{j\in\Pred_{\pi}(i)}V_j\Bigr)\Bigr],
\end{equation}
where $\pi$ is uniformly distributed on $\mathfrak{S}_n$.
This allocated share need not be an integer.
For instance, in Example~\ref{ex:three-lines}, the first
two lines in any ordering each contribute one new dimension,
whereas the last contributes none. Each line therefore
receives a share of $2/3$ of the total dimension $2$.
\end{remark}

\begin{example}\label{ex:polymatroid-Shapley}
Let $e_1,e_2,e_3$ be the standard basis of $\mathbbm{F}^3$, and set
\[
V_1=\operatorname{span}\{e_1,e_2\},\qquad
V_2=\operatorname{span}\{e_1,e_3\},\qquad
V_3=\operatorname{span}\{e_1,e_2+e_3\}.
\]
Each plane has dimension $2$, any two span $\mathbbm{F}^3$, and all pairwise and triple intersections equal $\operatorname{span}\{e_1\}$. Ordinary inclusion--exclusion incorrectly gives $6-3+1=4$. For distinct $i,j,l\in[3]$, one has $\dim(V_i\cap V_j)=\dim(V_i\cap V_l)=1$ and $\dim\bigl(V_i\cap(V_j+V_l)\bigr)=2$. Since $c_{3,2}=1/6$ and $c_{3,3}=1/3$, the symmetric Grassmann-type formula gives
\[
\dim(V_1+V_2+V_3)=\sum_{i=1}^{3}
\Bigl[2-\Bigl(\frac16+\frac16+\frac23\Bigr)\Bigr]=3.
\]
In a uniformly random ordering, each plane contributes $2$, $1$, or $0$ new dimensions with equal probabilities. Proposition~\ref{prop:Shapley} identifies each Grassmann contribution with its Shapley share:
\[
\underbrace{2-\Bigl(\frac{1}{6}+\frac{1}{6}+\frac{2}{3}\Bigr)}_{\text{dimension minus averaged intersection loss}}=\underbrace{\frac{1}{3}\,(2+1+0)}_{\text{average marginal dimension}}=\phi_i(\mathfrak{r})=1.
\]
The rank function satisfies $\mathfrak{r}(\{i\})=2$ and $\mathfrak{r}(S)=3$ for $|S|\geq2$. Its independence and base polytopes \cite{Fujishige05} are
\begin{align*}
P(\mathfrak{r})&\coloneq
\Bigl\{x\in\mathbbm{R}_{\geq0}^{3}\;\Big|\;
\sum_{i\in S}x_i\leq\mathfrak{r}(S) \quad\forall\,S\subseteq[3]\Bigr\}
=\bigl\{x\in[0,2]^3\mid x_1+x_2+x_3\leq3\bigr\}, \\
\mathcal{B}(\mathfrak{r})
&=\bigl\{x\in[0,2]^3\mid x_1+x_2+x_3=3\bigr\}.
\end{align*}
The six vertices of $\mathcal{B}(\mathfrak{r})$ are the ordered marginal allocations, the permutations of $(2,1,0)$. Their centre $(1,1,1)$ is the Shapley allocation, whose coordinate sum recovers the dimension $3$ by Shapley efficiency; see Figure~\ref{fig:polymatroid-Shapley}.
\begin{figure}[t]
\centering
\includegraphics[width=\textwidth]{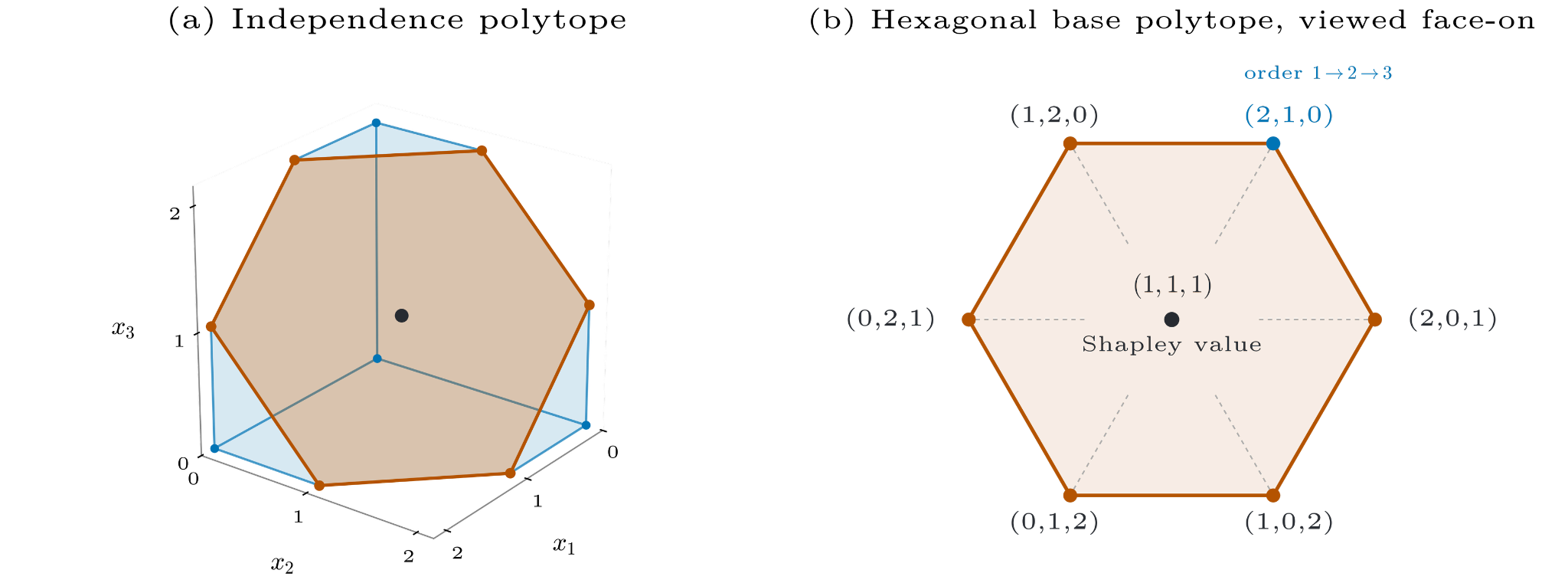}
\caption{Grassmann--Shapley correspondence for three planes with a common line. (a) The independence polytope $P(\mathfrak{r})$, with its base $\mathcal{B}(\mathfrak{r})$ in orange. (b) The six ordered allocations form a hexagon with the Shapley value allocation $(1,1,1)$ at its centre.}
\label{fig:polymatroid-Shapley}
\end{figure}
\end{example}

%%%%%%%%%%%%%%%%%%%%%%%%%%%%%%%%%%%%%%%%%%%%%%%%%
\section{Graded dimension-defect hierarchy} \label{sec:defect}
%%%%%%%%%%%%%%%%%%%%%%%%%%%%%%%%%%%%%%%%%%%%%%%%%
Theorem~\ref{thm:symmetric-Grassmann} naturally decomposes the failure of dimension additivity as a sum of intersection losses. Grouping these corrections by the number of subspaces involved refines the total dimension defect into a graded profile: the total defect records how much dimension is lost, while the profile shows at which interaction scales that loss becomes visible. We now define this profile and study the information it retains.

\begin{definition}[Dimension defect]\label{def:defect}
For $V_1,\ldots,V_n\subseteq V$, define
\begin{equation}\label{defect-dim}
\defop(V_1,\ldots,V_n)\coloneq\sum_{i\in[n]}\dim(V_i)-
\dim\Bigl(\sum_{i\in[n]}V_i\Bigr).
\end{equation}
For $2\leq k\leq n$, define the $k$-th graded interaction defect by
\begin{equation}\label{kth-defect-Ak}
A_k(V_1,\ldots,V_n)\coloneq c_{n,k} \sum_{\substack{T\subseteq[n]\\|T|=k}}
\sum_{i\in T}\dim\Bigl(V_i\cap\sum_{j\in T\setminus\{i\}}V_j\Bigr).
\end{equation}
\end{definition}

Thus $A_k$ is the average overlap between one distinguished subspace and the sum of $k-1$ other subspaces, taken over all such choices. In particular, $A_2$ is the average pairwise overlap, while $A_k$ for $k>2$ captures overlap with sums of several subspaces. The quantities $A_k$ are indexed by the number of subspaces involved, but they do not isolate interactions involving exactly $k$ subspaces; the same overlap may contribute to more than one $A_k$.

\begin{corollary}[Exact defect decomposition] \label{cor:defect-decomposition}
For every finite family $V_1,\ldots,V_n\subseteq V$,
\begin{equation}\label{defect-dim-decomposition}
\defop(V_1,\ldots,V_n)=\sum_{k=2}^{n}A_k(V_1,\ldots,V_n).
\end{equation}
In particular, each contribution is nonnegative.
\end{corollary}
\begin{proof}
Reindexing Theorem~\ref{thm:symmetric-Grassmann} by $T=S\cup\{i\}$, writing $k=|T|$, and using \eqref{defect-dim} and \eqref{kth-defect-Ak} gives \eqref{defect-dim-decomposition}.
\end{proof}

The decomposition \eqref{defect-dim-decomposition} retains information that the total defect alone does not. For example, in Example~\ref{ex:three-lines} all pairwise intersections vanish, so $A_2=0$, whereas $A_3=1$: the dependence is visible only when a subspace is compared with the sum of the other two. Thus the profile distinguishes pairwise overlap from dependence that appears only when more subspaces are considered.

\begin{remark}[Expected overlap]
Averaging the ordered identity \eqref{ordered-Grassmann} over uniformly random permutations gives
\begin{equation}
\defop(V_1,\ldots,V_n)=\mathtt{E}\Bigl[\sum_{i\in[n]}
\dim\Bigl(V_i\cap\sum_{j\in\Pred_{\pi}(i)}V_j\Bigr)\Bigr].
\end{equation}
Thus the dimension defect is the expected total overlap encountered when the subspaces are inserted in a uniformly random order.
\end{remark}

The graded defects satisfy an additional monotonicity property. At the rank-function level this is an averaged diminishing-returns statement for submodular functions \cite{Fujishige05}; we give the direct subspace proof.

\begin{theorem}[Monotonicity of the graded defects] \label{thm:Ak-monotone}
For every family $V_1,\ldots,V_n\subseteq V$,
\begin{equation}\label{kth-defect-Ak-monotone}
A_2\leq A_3\le\cdots\leq A_n.
\end{equation}
\end{theorem}
\begin{proof}
Fix $i\in[n]$ and define $f_i(S)\coloneq\dim(V_i\cap\sum_{j\in S}V_j)$ for $S\subseteq[n]\setminus\{i\}$. The function $f_i$ is monotone under inclusion. Let $m=n-1$, and for $0\leq s\leq m$ let
\begin{equation*}
M_{i,s}\coloneq\frac{1}{\binom{m}{s}}
\sum_{\substack{S\subseteq[n]\setminus\{i\}\\|S|=s}}f_i(S).
\end{equation*}
We claim that $M_{i,s}\leq M_{i,s+1}$ for $0\leq s\leq m-1$. Double count pairs $(S,T)$ with $S\subset T$, $|S|=s$, and $|T|=s+1$. Since $f_i(S)\leq f_i(T)$, then $(m-s)\sum_{|S|=s}f_i(S)\leq(s+1)\sum_{|T|=s+1}f_i(T)$. Using $(m-s)\binom{m}{s}=(s+1)\binom{m}{s+1}$ gives $M_{i,s}\leq M_{i,s+1}$. Finally, since $k\binom{n}{k}=n\binom{n-1}{k-1}$, we may rewrite \eqref{kth-defect-Ak} as $A_k=\frac{1}{n}\sum_{i\in[n]}M_{i,k-1}$. For $2\leq k\leq n-1$, averaging the inequalities $M_{i,k-1}\leq M_{i,k}$ over $i$ proves \eqref{kth-defect-Ak-monotone}.
\end{proof}

The monotonicity has a direct averaged interpretation. For a fixed $i$, the quantity $M_{i,s}$ is the average dimension of the part of $V_i$ already contained in the span of a collection of $s$ other subspaces. Adding subspaces to that collection can only increase the contained part, and averaging over $i$ gives $A_2\leq\cdots\leq A_n$. Equivalently, as the span available for comparison grows, the average overlap with a distinguished subspace cannot decrease.

\begin{corollary}[Defect bounds]\label{cor:defect-bounds}
Let $p\coloneq\frac{1}{\binom{n}{2}}\sum_{1\leq i<j\leq n}\dim(V_i\cap V_j)$ be the average pairwise intersection dimension, and let $q\coloneq\frac{1}{n}\sum_{i\in[n]}\dim(V_i\cap\sum_{j\neq i}V_j)$ be the average dimension of the intersection of each subspace with the sum of all the others. Then
\begin{equation}\label{defect-dim-bounds}
(n-1)p\leq\defop(V_1,\ldots,V_n)\leq(n-1)q.
\end{equation}
\end{corollary}
\begin{proof}
By definition, $A_2=p$ and $A_n=q$. Hence Theorem~\ref{thm:Ak-monotone} and Corollary~\ref{cor:defect-decomposition} imply
\begin{equation*}
(n-1)A_2\le\sum_{k=2}^{n}A_k\le(n-1)A_n.
\end{equation*}
\end{proof}

The next statement includes the standard direct-sum criterion \cite[Lemma~2.1]{MS25} and expresses it in terms of the defect profile.

\begin{corollary}[Direct-sum criterion]\label{cor:direct-sum}
The following are equivalent:
\begin{enumerate}[label=\textnormal{(\roman*)}]
\item $\dim(\sum_{i\in[n]}V_i)=\sum_{i\in[n]}\dim(V_i)$;
\item $A_k=0$ for every $2\leq k\leq n$;
\item $A_n=0$;
\item $V_i\cap\sum_{j\neq i}V_j=\{0\}$ for every $i\in[n]$;
\item $\sum_{i\in[n]}V_i=\bigoplus_{i\in[n]}V_i$.
\end{enumerate}
\end{corollary}
\begin{proof}
The equivalence of (i) and (ii) follows from \eqref{defect-dim-decomposition} and nonnegativity. By monotonicity, (ii) and (iii) are equivalent. Since $A_n$ is the average of the nonnegative numbers
$\dim(V_i\cap\sum_{j\neq i}V_j)$, conditions (iii) and (iv) are equivalent. Finally, (iv) is the standard characterization of an internal direct sum, giving (iv)$\Leftrightarrow$(v).
\end{proof}

The symmetric formula can be useful even when only coarse intersection data are available. As already used in Corollary~\ref{cor:direct-sum}, the condition $V_i\cap\sum_{j\ne i}V_j=\{0\}$ forces every smaller loss associated with $i$ to vanish, by inclusion. Pairwise disjointness alone gives only $A_2=0$; Example~\ref{ex:three-lines} shows that $A_3$ can still be positive. Thus the remaining terms of the defect decomposition detect dependence missed by pairwise intersections.

\begin{remark}\label{rem:n4}
For $n=4$, the coefficients in Theorem~\ref{thm:symmetric-Grassmann} are $\frac{1}{12}$ for $|S|=1$, $\frac{1}{12}$ for $|S|=2$, and $\frac{1}{4}$ for $|S|=3$. Thus, $\defop(V_1,V_2,V_3,V_4)=A_2+A_3+A_4$, where $A_2\leq A_3\leq A_4$. The equality separates pairwise overlap, three-subspace overlap with a distinguished member, and one-versus-all-others overlap without introducing alternating inclusion--exclusion terms.
\end{remark}

%%%%%%%%%%%%%%%%%%%%%%%%%%%%%%%%%%%%%%%%%%%%%%%%%
\subsection{Refinements of the defect hierarchy}
\label{subsec:refinements}
%%%%%%%%%%%%%%%%%%%%%%%%%%%%%%%%%%%%%%%%%%%%%%%%%
The preceding results resolve the failure of dimension additivity into the graded quantities $A_2,\ldots,A_n$. This refinement is different in nature from the dimension formulas of Tian \cite{Tian19} and Makrides--Szemberg \cite{MS25}, which determine the total dimension correction through global sum and intersection data. From the polymatroid viewpoint, monotonicity of marginal losses is closely related to the diminishing-returns property of submodular rank functions; see, for example, \cite{Fujishige05}. Moreover, linearly representable polymatroids satisfy additional rank inequalities that need not hold for arbitrary polymatroids, as illustrated by the inequalities for subspace arrangements obtained in \cite{Kinser11}. In the present setting, the symmetric averages defining $A_k$ yield further exact relations. We derive below exact remainder terms for the defect bounds, characterize their equality cases, relate the hierarchy to the average rank profile of the subspace family, and obtain a dual hierarchy for intersections.

For $2\leq k\leq n-1$, define the successive defect gaps by
\begin{equation}\label{defect-gap-Bk}
B_k\coloneq A_{k+1}-A_k.
\end{equation}
By Theorem~\ref{thm:Ak-monotone}, one has $B_k\ge0$. Thus $B_k$ measures the increase in the average overlap when the interaction order is raised from $k$ to $k+1$.

\begin{proposition}\label{prop:Bk-local}
For $2\leq k\leq n-1$,
\begin{equation}\label{defect-gap-Bk-local}
B_k=\frac{c_{n,k}}{n-k}
\sum_{i\in[n]}\sum_{\substack{S\subseteq[n]\setminus\{i\}\\|S|=k-1}}
\sum_{j\in[n]\setminus(S\cup\{i\})}\Bigl[\dim\Bigl(V_i\cap\Bigl(V_j+\sum_{l\in S}V_l\Bigr)\Bigr)-\dim\Bigl(V_i\cap\sum_{l\in S}V_l\Bigr)\Bigr].
\end{equation}
In particular, $B_k$ is the average additional overlap created by adjoining one further subspace to a $(k-1)$-element predecessor family.
\end{proposition}
\begin{proof}
Fix $i\in[n]$. In the sum in \eqref{defect-gap-Bk-local}, each $k$-element set $T\subseteq[n]\setminus\{i\}$ occurs exactly $k$ times as $T=S\cup\{j\}$, once for each choice of $j\in T$. Each $(k-1)$-element set $S\subseteq[n]\setminus\{i\}$ occurs exactly $n-k$ times in the subtracted term, once for each admissible choice of $j$. Hence, using the definitions of $f_i$ and $M_{i,s}$ from the proof of Theorem~\ref{thm:Ak-monotone}, the sum over $S$ and $j$ for this fixed $i$ is
\begin{equation*}
k\binom{n-1}{k}M_{i,k}
-(n-k)\binom{n-1}{k-1}M_{i,k-1}=(n-k)\binom{n-1}{k-1}
(M_{i,k}-M_{i,k-1}),
\end{equation*}
where $k\binom{n-1}{k}=(n-k)\binom{n-1}{k-1}$. Multiplying by $c_{n,k}/(n-k)$ and summing over $i$ gives
\begin{equation*}
\frac{1}{n}\sum_{i\in[n]}(M_{i,k}-M_{i,k-1})
=A_{k+1}-A_k=B_k,
\end{equation*}
because $c_{n,k}=1/(n\binom{n-1}{k-1})$ and $A_\ell=\frac{1}{n}\sum_{i\in[n]}M_{i,\ell-1}$. This proves \eqref{defect-gap-Bk-local}.
\end{proof}

Thus $B_k$ is the average increase in the overlap of the distinguished subspace $V_i$ with the span of the other subspaces when one further subspace $V_j$ is adjoined. The initial configuration consists of $V_i$ and the $k-1$ members of $S$, so its interaction order is $k$; adjoining $V_j$ raises it to $k+1$.

For $S\subseteq[n]$, set $U_S\coloneq\sum_{i\in S}V_i$, and for $j\in[n]\setminus S$ define
\begin{equation}\label{quotient-subspaces}
\overline{V}_j^{\,S}\coloneq\frac{V_j+U_S}{U_S}\subseteq\frac{V}{U_S}.
\end{equation}

\begin{corollary}[Quotient intersections and vanishing curvature]\label{cor:Bk-quotient}
For every $2\leq k\leq n-1$,
\begin{equation}\label{Bk-quotient}
B_k=\frac{1}{\binom{n}{k-1}\binom{n-k+1}{2}}
\sum_{\substack{S\subseteq[n]\\|S|=k-1}}
\sum_{\substack{j<l\\j,l\in[n]\setminus S}}
\dim(\overline{V}_j^{\,S}\cap\overline{V}_l^{\,S}).
\end{equation}
Consequently, $B_k=0$ if and only if
\begin{equation}\label{Bk-zero-quotient}
(V_j+U_S)\cap(V_l+U_S)=U_S
\end{equation}
for every $S\subseteq[n]$ with $|S|=k-1$ and every pair of
distinct indices $j,l\in[n]\setminus S$.
\end{corollary}
\begin{proof}
Let $Q_S:V\to V/U_S$ be the quotient map. Its restriction to $V_j\cap(V_l+U_S)$ has image $\overline{V}_j^{\,S}\cap\overline{V}_l^{\,S}$ and kernel $V_j\cap U_S$. Hence
\[
\dim(\overline{V}_j^{\,S}\cap\overline{V}_l^{\,S})
=\dim\bigl(V_j\cap(V_l+U_S)\bigr)-\dim(V_j\cap U_S).
\]
Substitute this identity into Proposition~\ref{prop:Bk-local}. Each unordered pair occurs twice in its ordered-pair sum, and
\[
n\binom{n-1}{k-1}(n-k)=2\binom{n}{k-1}\binom{n-k+1}{2}.
\]
This gives \eqref{Bk-quotient}. Since every summand is nonnegative, $B_k=0$ precisely when all quotient intersections vanish. Taking their inverse images under $Q_S$ gives \eqref{Bk-zero-quotient}.
\end{proof}

Following we show that the successive gaps provide exact refinements of both inequalities in Corollary~\ref{cor:defect-bounds}.

\begin{theorem}[Exact gap formulas]\label{thm:exact-gap}
For every finite family of subspaces $V_1,\ldots,V_n\subseteq V$ with $n\geq2$,
\begin{equation}\label{lower-gap-exact}
\defop(V_1,\ldots,V_n)=(n-1)A_2+\sum_{k=2}^{n-1}(n-k)B_k,
\end{equation}
and
\begin{equation}\label{upper-gap-exact}
\defop(V_1,\ldots,V_n)=(n-1)A_n-\sum_{k=2}^{n-1}(k-1)B_k.
\end{equation}
Thus the errors in the two endpoint estimates of Corollary~\ref{cor:defect-bounds} are themselves resolved into nonnegative contributions indexed by interaction order.
\end{theorem}
\begin{proof}
Substitute $A_k=A_2+\sum_{j=2}^{k-1}B_j$ into $\defop=\sum_{k=2}^n A_k$. The gap $B_j$ occurs $n-j$ times, giving \eqref{lower-gap-exact}. Similarly, substituting $A_k=A_n-\sum_{j=k}^{n-1}B_j$ counts $B_j$ exactly $j-1$ times and gives \eqref{upper-gap-exact}.
\end{proof}

The identities \eqref{lower-gap-exact}-\eqref{upper-gap-exact} give the endpoint bounds exact, scale-by-scale remainders. In particular, $\defop(V_1,\ldots,V_n)-(n-1)A_2$ measures precisely what the pairwise lower estimate misses, while $(n-1)A_n-\defop(V_1,\ldots,V_n)$ measures the gap to the one-versus-all upper estimate. Each $B_k$ identifies the contribution to that gap associated with a particular increase in interaction order.

The exact remainder terms lead to a rigidity statement for equality in either defect bound.

\begin{theorem}[Common-core rigidity]\label{thm:common-core}
Let $V_1,\ldots,V_n\subseteq V$ and set $D\coloneq\bigcap_{i\in[n]}V_i$. The following conditions are equivalent:
\begin{enumerate}[label=\textnormal{(\roman*)}]
\item $\defop(V_1,\ldots,V_n)=(n-1)A_2$;

\item $\defop(V_1,\ldots,V_n)=(n-1)A_n$;

\item $A_2=A_3=\cdots=A_n$;

\item $V_i\cap\sum_{j\neq i}V_j=D$ for every $i\in[n]$;

\item the quotient subspaces $V_1/D,\ldots,V_n/D\subseteq V/D$ form an internal direct sum;

\item for every $i\in[n]$ there exists a complement $U_i$ of $D$ in $V_i$ such that $D\oplus U_1\oplus\cdots\oplus U_n$ is a direct sum.
\end{enumerate}
Whenever these conditions hold,
\begin{equation}\label{kth-defect-Ak-common-core}
A_2=\cdots=A_n=\dim(D)
\end{equation}
and
\begin{equation}\label{def-common-core}
\defop(V_1,\ldots,V_n)=(n-1)\dim(D).
\end{equation}
\end{theorem}
\begin{proof}
By Theorem~\ref{thm:exact-gap}, equality in either endpoint bound holds if and only if $B_k=0$ for every $2\leq k\leq n-1$. Hence conditions (i), (ii), and (iii) are equivalent. Assume (iii). Using $M_{i,s}$ and its monotonicity from the proof of Theorem~\ref{thm:Ak-monotone}, $A_n-A_2=0$ forces $M_{i,n-1}=M_{i,1}$ for every $i$. For fixed $i$ and every $j\neq i$, each $\dim(V_i\cap V_j)$ is at most $\dim(V_i\cap\sum_{k\ne i}V_k)=M_{i,n-1}$, while their average is $M_{i,1}=M_{i,n-1}$. Hence every one of these dimensions equals $M_{i,n-1}$, and containment gives $V_i\cap V_j=V_i\cap\sum_{k\ne i}V_k$ for all $j\ne i$. Since $V_i\cap V_j=V_j\cap V_i$, these spaces have a common value for all $i$, and that value is $D=\bigcap_{i\in[n]}V_i$. This proves (iv). Condition (iv) is equivalent to $(V_i/D)\cap\sum_{j\neq i}(V_j/D)=\{0\}$ for every $i$, which is the standard direct-sum criterion for the quotient family. Thus (iv) and (v) are equivalent. Choosing complements $U_i\subseteq V_i$ such that $V_i=D\oplus U_i$ shows that (v) and (vi) are equivalent. Finally, assume (iv). For every nonempty $S\subseteq[n]\setminus\{i\}$, we have $D\subseteq V_i\cap\sum_{j\in S}V_j\subseteq V_i\cap\sum_{j\neq i}V_j=D$. Thus every intersection dimension defining $A_k$ equals $\dim(D)$, so $A_2=\cdots=A_n=\dim(D)$. This proves (iii), completes the equivalence, and establishes \eqref{kth-defect-Ak-common-core}. Equation~\eqref{def-common-core} follows from Corollary~\ref{cor:defect-decomposition}.
\end{proof}

The Theorem~\ref{thm:common-core} says that equality in either endpoint bound occurs exactly
when the defect profile is flat: no new average overlap appears as the
interaction order increases. Geometrically, all overlap is then accounted
for by the common subspace $D$. After quotienting by $D$, the remaining
subspaces form an internal direct sum, so there is no further dependence
among them.

\begin{remark}
For families of subspaces of equal dimension, condition~\textnormal{(iv)} is the defining property of a sunflower of maximal dimension in \cite[Section~2]{BSSV17}. For such families with constant pairwise intersection dimension, a characterization by successive rank increments is given in \cite[Remark~1.3]{LST22}. Theorem~\ref{thm:common-core} characterizes the common-core decomposition for arbitrary finite subspace families through equality in either averaged defect bound.
\end{remark}

The hierarchy also admits a useful description purely in terms of average ranks of collections of subspaces.

\begin{definition}[Average rank profile]
\label{def:average-rank}
For $0\leq k\leq n$, define
\begin{equation}
R_k(V_1,\ldots,V_n)\coloneq\frac{1}{\binom{n}{k}}
\sum_{\substack{S\subseteq[n]\\|S|=k}}
\dim\Bigl(\sum_{i\in S}V_i\Bigr),
\end{equation}
where $R_0=0$.
\end{definition}

Thus $R_k$ is the average dimension of the span of a uniformly chosen collection of $k$ subspaces, or equivalently the average of $\mathfrak{r}(S)$ over all $k$-element subsets $S\subseteq[n]$. Its increment $\Delta_k\coloneq R_k-R_{k-1}$ is the average marginal increase in span dimension as the number of selected subspaces grows from $k-1$ to $k$. Submodularity implies that these average increments decrease with $k$, so the profile $(R_k)_{k=0}^n$ is discretely concave \cite{Fujishige05}. Moreover, $(R_k)_{k=0}^n$ is the cardinality profile of the full permutation average of $\mathfrak{r}$, placing this construction in the general framework of permutation-symmetric polymatroids and entropy functions \cite{ChenYeung16}. The following theorem makes the diminishing-returns interpretation precise by identifying $A_k=R_1-\Delta_k$ and $B_k=\Delta_k-\Delta_{k+1}$. Thus $A_k$ measures the loss in average marginal dimension relative to the first subspace, while $B_k$ measures the decrease in that marginal dimension from one level to the next.

\begin{theorem}[Average-rank representation]
\label{thm:average-rank}
For every $2\leq k\leq n$,
\begin{equation}\label{kth-defect-Ak-Rk}
A_k=R_1+R_{k-1}-R_k.
\end{equation}
Equivalently, if
\begin{equation}
\Delta_k\coloneq R_k-R_{k-1}, \qquad 1\leq k\leq n,
\end{equation}
then
\begin{equation}\label{kth-defect-Ak-Delta}
A_k=R_1-\Delta_k.
\end{equation}
Moreover, for $2\leq k\leq n-1$,
\begin{equation}\label{defect-gap-Bk-Rk}
B_k=2R_k-R_{k-1}-R_{k+1}=\Delta_k-\Delta_{k+1}.
\end{equation}
Consequently,
\begin{equation}\label{Rk-concave}
R_{k+1}-R_k\leq R_k-R_{k-1}, \qquad 1\leq k\leq n-1,
\end{equation}
so the average rank profile $(R_k)_{k=0}^{n}$ is discretely concave.
\end{theorem}
\begin{proof}
For $T\subseteq[n]$ with $|T|=k$ and $i\in T$,
the classical Grassmann formula gives
\begin{equation}\label{pointed-rank}
\dim\Bigl(V_i\cap\sum_{j\in T\setminus\{i\}}V_j\Bigr)=
\dim(V_i)+\dim\Bigl(\sum_{j\in T\setminus\{i\}}V_j\Bigr)
-\dim\Bigl(\sum_{j\in T}V_j\Bigr).
\end{equation}
Average \eqref{pointed-rank} over all pairs $(T,i)$ with $|T|=k$ and $i\in T$. The first term averages to $R_1$, the second to $R_{k-1}$, and the third to $R_k$. This proves \eqref{kth-defect-Ak-Rk}. Equation~\eqref{kth-defect-Ak-Delta} follows immediately. Taking the difference of \eqref{kth-defect-Ak-Rk} at levels $k+1$ and $k$ gives $B_k=A_{k+1}-A_k=2R_k-R_{k-1}-R_{k+1}=\Delta_k-\Delta_{k+1}$. Since $B_k\ge0$, one has $\Delta_2\geq\cdots\geq\Delta_n$. The remaining first step follows from $\Delta_1-\Delta_2=A_2\geq0$. This proves \eqref{Rk-concave}.
\end{proof}

\begin{remark}\label{rem:rank-curvature}
Equation~\eqref{defect-gap-Bk-Rk} shows that $B_k$ is the discrete curvature of the average-rank profile: $B_k=\Delta_k-\Delta_{k+1}$. It measures how much the average marginal dimension drops as the number of selected subspaces increases from $k$ to $k+1$. Hence $B_k=0$ means that the average marginal dimension is unchanged across those two steps, while $B_k>0$ records a drop at that scale. Thus the nonnegativity and monotonicity of the defect hierarchy, $0\leq A_2\leq\cdots\leq A_n$, are equivalent to the discrete concavity of $(R_k)_{k=0}^{n}$: $A_k$ measures average overlap at interaction order $k$, whereas $B_k$ measures the loss of average marginal rank when one passes from $k$ to $k+1$ participating subspaces. Notice also that
\begin{equation}
\defop(V_1,\ldots,V_n)=nR_1-R_n.
\end{equation}
Under the equivalent conditions of Theorem~\ref{thm:common-core}, with common core $D$, one has
\begin{equation}
R_k=kR_1-(k-1)\dim(D), \qquad 1\leq k\leq n.
\end{equation}
\end{remark}

\begin{example}\label{ex:same-defect-different-profiles}
Let $e_1,e_2,e_3$ be the standard basis of $\mathbbm{F}^3$ and consider the two arrangements
\begin{align*}
U&\coloneq\bigl(\operatorname{span}\{e_1\},
\operatorname{span}\{e_1+e_2\},
\operatorname{span}\{e_2\},
\operatorname{span}\{e_3\}\bigr), \\
W&\coloneq\bigl(\operatorname{span}\{e_1\},
\operatorname{span}\{e_2\},
\operatorname{span}\{e_3\},
\operatorname{span}\{e_1+e_2+e_3\}\bigr).
\end{align*}
Both have total dimension defect $1$. Their average-rank profiles are
\[
(R_0,R_1,R_2,R_3,R_4)(U)=\Bigl(0,1,2,\frac{11}{4},3\Bigr),
\qquad
(R_0,R_1,R_2,R_3,R_4)(W)=(0,1,2,3,3).
\]
Theorem~\ref{thm:average-rank} therefore gives
\[
(A_2,A_3,A_4)(U)=\Bigl(0,\frac{1}{4},\frac{3}{4}\Bigr),
\qquad
(A_2,A_3,A_4)(W)=(0,0,1).
\]
The successive gaps $B_k=A_{k+1}-A_k$ are
\[
(B_2,B_3)(U)=\Bigl(\frac{1}{4},\frac{1}{2}\Bigr),
\qquad
(B_2,B_3)(W)=(0,1).
\]
Thus $U$ has positive curvature at both levels, whereas the curvature of $W$ is concentrated at the final level. Figure~\ref{fig:rank-defect-profiles} compares the average-rank and dimension-defect profiles of these two arrangements. The defect profile distinguishes dependence visible at intermediate levels from dependence visible only at the final level, even when the individual dimensions and total defect agree.
\begin{figure}[t]
\centering
\includegraphics[width=\textwidth]{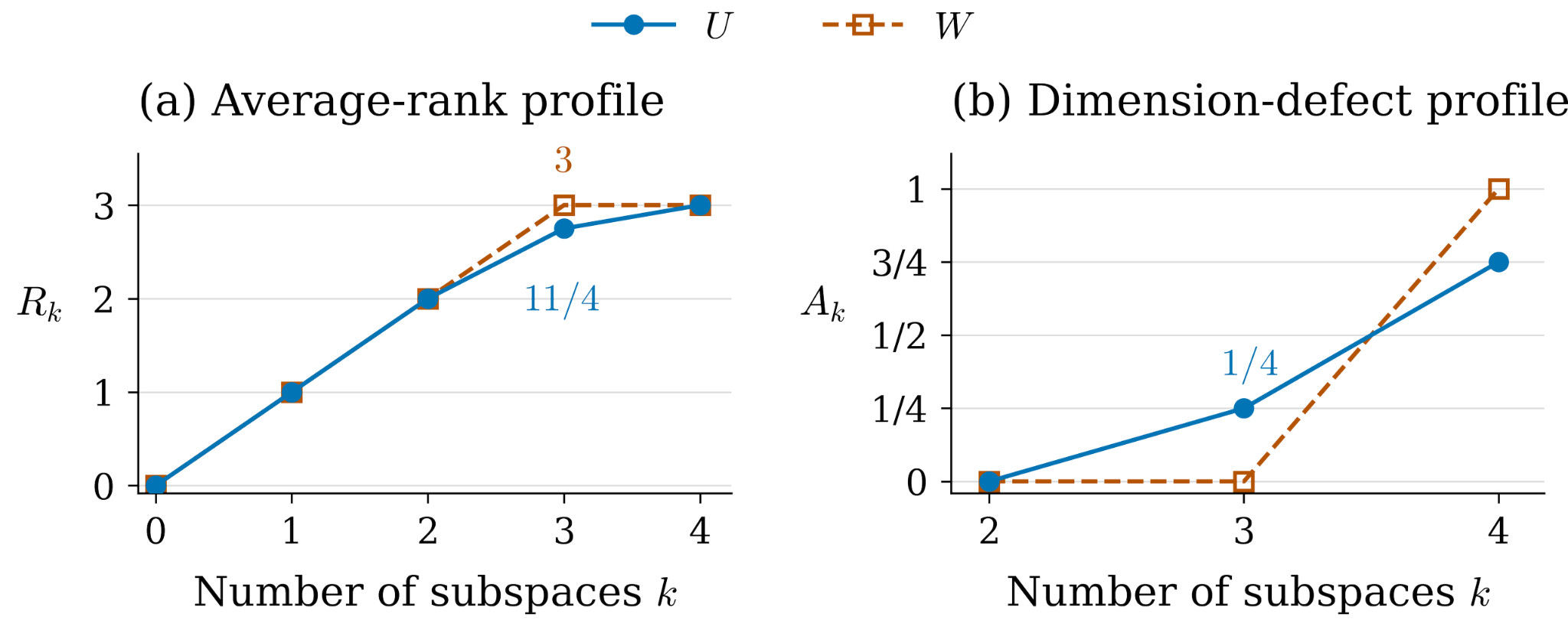}
\caption{Average-rank profiles (a) and dimension-defect profiles (b) of the arrangements $U$ and $W$ in Example~\ref{ex:same-defect-different-profiles}. Blue solid lines with circles represent $U$; orange dashed lines with squares represent $W$. Both arrangements have total dimension defect $1$, but their profiles distinguish dependence at intermediate levels from dependence appearing only at the final level.} \label{fig:rank-defect-profiles}
\end{figure}
\end{example}

We finally record the dual hierarchy. This strengthens the dual dimension
formula of Theorem~\ref{thm:dual-Grassmann}: the entire graded decomposition and its
monotonicity admit dual counterparts.

Let $d=\dim(V)$ and write $\operatorname{codim}(U)\coloneq d-\dim(U)$ for every subspace $U\subseteq V$.

\begin{definition}[Dual dimension defect]
\label{def:dual-defect}
Define
\begin{equation}
\defop^{\vee}(V_1,\ldots,V_n)\coloneq
\sum_{i\in[n]}\operatorname{codim}(V_i)-
\operatorname{codim}\Bigl(\bigcap_{i\in[n]}V_i\Bigr).
\end{equation}
For $2\leq k\leq n$, define
\begin{equation}
A_k^{\vee}\coloneq c_{n,k}
\sum_{\substack{T\subseteq[n]\\|T|=k}}
\sum_{i\in T}\operatorname{codim}\Bigl(V_i+\bigcap_{j\in T\setminus\{i\}}V_j\Bigr).
\end{equation}
\end{definition}

\begin{theorem}[Dual defect hierarchy]
\label{thm:dual-Grassmann-hierarchy}
For every finite family $V_1,\ldots,V_n\subseteq V$,
\begin{equation}
\defop^{\vee}(V_1,\ldots,V_n)=\sum_{k=2}^{n}A_k^{\vee},
\end{equation}
and
\begin{equation}
A_2^{\vee}\leq A_3^{\vee}\leq\cdots\leq A_n^{\vee}.
\end{equation}
\end{theorem}
\begin{proof}
Apply Corollary~\ref{cor:defect-decomposition} and
Theorem~\ref{thm:Ak-monotone} to $V_1^\circ,\ldots,V_n^\circ$.
The annihilator identities \eqref{annihilator} identify their total defect
with $\defop^\vee$ and their $k$-th graded defect with $A_k^\vee$, since
\begin{equation*}
\Bigl(V_i+\bigcap_{j\in T\setminus\{i\}}V_j\Bigr)^\circ
=V_i^\circ\cap\sum_{j\in T\setminus\{i\}}V_j^\circ.
\end{equation*}
Both assertions follow.
\end{proof}

\begin{corollary}[Dual defect bounds]\label{cor:dual-bounds}
Set $p^{\vee}\coloneq\frac{1}{\binom{n}{2}}\sum_{1\leq i<j\leq n}\operatorname{codim}(V_i+V_j)$ and $q^{\vee}\coloneq\frac{1}{n}\sum_{i\in[n]}\operatorname{codim}(V_i+\bigcap_{j\neq i}V_j)$. Then
\begin{equation}\label{dual-bounds}
(n-1)p^{\vee}\leq\defop^{\vee}(V_1,\ldots,V_n)\leq(n-1)q^{\vee}.
\end{equation}
\end{corollary}
\begin{proof}
By definition, $A_2^{\vee}=p^{\vee}$ and $A_n^{\vee}=q^{\vee}$. The bound \eqref{dual-bounds} follows immediately from Theorem~\ref{thm:dual-Grassmann-hierarchy}.
\end{proof}
\begin{corollary}[Exact dual gap formulas]
If $B_k^{\vee}\coloneq A_{k+1}^{\vee}-A_k^{\vee}$ for $2\leq k\leq n-1$, then
\begin{align}\label{dual-lower-gap}
\defop^{\vee}(V_1,\ldots,V_n)&=(n-1)A_2^{\vee}+\sum_{k=2}^{n-1}(n-k)B_k^{\vee},
\\ \label{dual-upper-gap}
\defop^{\vee}(V_1,\ldots,V_n)&=(n-1)A_n^{\vee}-\sum_{k=2}^{n-1}(k-1)B_k^{\vee}.
\end{align}
\end{corollary}
\begin{proof}
The exact identities \eqref{dual-lower-gap} and \eqref{dual-upper-gap} follow from the same telescoping argument used in Theorem~\ref{thm:exact-gap}.
\end{proof}

\begin{corollary}[Equality in the dual defect bounds]
\label{cor:dual-equality}
Set $T\coloneq\sum_{i\in[n]}V_i$. Equality in either bound of Corollary~\ref{cor:dual-bounds} holds if and only if
\begin{equation}\label{dual-equality}
V_i+\bigcap_{j\neq i}V_j=T
\qquad\forall~i\in[n].
\end{equation}
In this case, $A_2^\vee=\cdots=A_n^\vee=\operatorname{codim}(T)$, and
\begin{equation}
\defop^\vee(V_1,\ldots,V_n)=(n-1)\operatorname{codim}(T).
\end{equation}
\end{corollary}
\begin{proof}
Apply Theorem~\ref{thm:common-core} to $V_1^\circ,\ldots,V_n^\circ$. Their common core is $T^\circ$, and the equality criterion is $V_i^\circ\cap\sum_{j\neq i}V_j^\circ=T^\circ$ for every $i\in[n]$. The annihilator identities \eqref{annihilator} translate this condition into \eqref{dual-equality}. The remaining assertions follow from $\dim(T^\circ)=\operatorname{codim}(T)$ and the common-core formulas.
\end{proof}

The original and dual hierarchies therefore provide complementary descriptions of dependence in a finite family of subspaces. The former measures the progressive redundancy encountered under sums, while the latter measures the corresponding codimension loss under intersections. In both cases the successive gaps yield exact remainder terms, refining the endpoint inequalities by resolving their errors according to interaction order.

%%%%%%%%%%%%%%%%%%%%%%%%%%%%%%%%%%%%%%%%%%%%%%%%%
\subsection{Representability-specific refinements and Kinser slacks}
\label{subsec:Kinser}
%%%%%%%%%%%%%%%%%%%%%%%%%%%%%%%%%%%%%%%%%%%%%%%%%
The nonnegativity of the gaps $B_k$ follows from submodularity and therefore holds for arbitrary polymatroid rank functions. This raises a further question: does the same curvature profile also record constraints that depend on linear representability? Representable polymatroids satisfy additional inequalities, beginning with Ingleton's inequality \cite{Ingleton71}. More generally, Kinser \cite{Kinser11} constructed, for every $n\geq4$, an irreducible inequality for rank functions of arrangements of $n$ subspaces; for $n=4$ it reduces to Ingleton's inequality. This subsection connects those representability-specific inequalities to the defect curvature $B_k$.

A related geometric viewpoint was developed by Friedman and Tootooni \cite{FT23}, who studied the failure of distributive behavior for subspaces through their notion of discoordination. In particular, for three subspaces $A,B,C$ the quantity
\begin{equation}
\dim\bigl((A+C)\cap(B+C)\bigr)-\dim\bigl((A\cap B)+C\bigr)
\end{equation}
measures precisely such a failure of distributivity and appears naturally in the proof of Ingleton-type inequalities. The purpose of this subsection is to connect these representability-specific inequalities with the graded defect hierarchy developed above. We first give an exact geometric decomposition of the full Kinser slack, and then show that permutation averages of Kinser slacks recover the successive defect gaps $B_k$.

For subspaces $X\subseteq Z$, write
\begin{equation}
[Z:X]\coloneq\dim(Z)-\dim(X).
\end{equation}

For $n\geq4$, define the \emph{Kinser slack} of the ordered family $(V_1,\ldots,V_n)$ by
\begin{align}\nonumber
\mathcal{K}_n(V_1,\ldots,V_n)\coloneq &~\dim(V_1+V_3)+\dim(V_1+V_n)+\dim(V_2+V_3) \\ \nonumber
&+\sum_{i=4}^{n}\Bigl[\dim(V_2+V_i)+\dim(V_{i-1}+V_i)\Bigr] \\ \nonumber
&-\dim(V_1+V_2)-\dim(V_1+V_3+V_n)-\dim(V_3) \\
&-\sum_{i=4}^{n}\Bigl[\dim(V_i)+\dim(V_2+V_{i-1}+V_i)\Bigr].
\label{Kinser-slack}
\end{align}
Kinser's inequality \cite{Kinser11} states that
\begin{equation}\label{Kinser-nonnegative}
\mathcal{K}_n(V_1,\ldots,V_n)\geq0.
\end{equation}
For $n=4$, \eqref{Kinser-nonnegative} is precisely Ingleton's inequality \cite{Ingleton71}.

The following result resolves the entire Kinser slack into explicit nonnegative quotient dimensions by retaining the remainders in the proof of \cite[Section~3.1]{Kinser11}.

\begin{theorem}[Geometric decomposition of the Kinser slack]
\label{thm:Kinser-slack-decomposition}
Let $V_1,\ldots,V_n\subseteq V$, with $n\ge4$, and set $K\coloneq\bigcap_{i=3}^{n}V_i$. Then
\begin{align}\nonumber
\mathcal{K}_n(V_1,\ldots,V_n)=&~[K+V_1+V_2:V_1+V_2]+[(V_1+V_3)\cap(V_1+V_n):V_1+K] \\
&+[(K\!+\!V_1)\cap(K\!+\!V_2):K]+
\sum_{i=4}^{n}\Bigl[(V_i\!+\!V_2)\cap(V_i\!+\!V_{i-1}):
V_i\!+\!\bigcap_{j=2}^{i-1}V_j\Bigr].
\label{Kinser-slack-decomposition}
\end{align}
In particular, every summand on the right-hand side is nonnegative.
\end{theorem}
\begin{proof}
Following the auxiliary spaces and descending intersection chain in \cite[Section~3.1]{Kinser11}, set $\mathcal{L}_0\coloneq[K+V_1+V_2:K+V_1]$ and $\mathcal{R}_0\coloneq[K+V_2:K]$. By the Grassmann formula applied to $K+V_1$ and $K+V_2$,
\begin{equation}\label{Kinser-middle-slack}
\mathcal{R}_0-\mathcal{L}_0=[(K+V_1)\cap(K+V_2):K].
\end{equation}
Since $V_1+V_2\subseteq K+V_1+V_2$ and $K+V_1\subseteq(V_1+V_3)\cap(V_1+V_n)$, we obtain the exact identity
\begin{align}\nonumber
\mathcal{L}_0=&\dim(V_1+V_2)+\dim(V_1+V_3+V_n)-\dim(V_1+V_3)-\dim(V_1+V_n) \\
&+[K+V_1+V_2:V_1+V_2]+[(V_1+V_3)\cap(V_1+V_n):V_1+K].
\label{Kinser-left-exact}
\end{align}
Indeed,
\begin{equation*}
\dim\bigl((V_1+V_3)\cap(V_1+V_n)\bigr)=\dim(V_1+V_3)+\dim(V_1+V_n)-\dim(V_1+V_3+V_n).
\end{equation*}
On the other hand,
\begin{equation}\label{Kinser-chain}
\mathcal{R}_0=[V_2:V_2\cap K]=\sum_{i=3}^{n}
\Bigl[\bigcap_{j=2}^{i-1}V_j:\bigcap_{j=2}^{i}V_j\Bigr].
\end{equation}
For $i=3$, the corresponding term is
\begin{equation}\label{Kinser-i3}
[V_2:V_2\cap V_3]=\dim(V_2+V_3)-\dim(V_3).
\end{equation}
For $i\geq4$, put $X_i\coloneq\bigcap_{j=2}^{i-1}V_j$. Then $[X_i:X_i\cap V_i]=[V_i+X_i:V_i]$. Moreover, $V_i+X_i\subseteq(V_i+V_2)\cap(V_i+V_{i-1})$, and therefore
\begin{align}\nonumber
[X_i:X_i\cap V_i]=&\dim(V_2+V_i)+\dim(V_{i-1}+V_i)-\dim(V_2+V_{i-1}+V_i)
-\dim(V_i) \\
&-[(V_i+V_2)\cap(V_i+V_{i-1}):V_i+X_i].
\label{Kinser-chain-exact}
\end{align}
Substituting \eqref{Kinser-i3} and \eqref{Kinser-chain-exact} into \eqref{Kinser-chain}, and then using \eqref{Kinser-middle-slack} and \eqref{Kinser-left-exact}, gives \eqref{Kinser-slack-decomposition}.
\end{proof}

Thus the Kinser slack is not merely nonnegative; it is given exactly by the sum of the displayed quotient dimensions. Each term measures a particular geometric failure of one subspace to contain, or coincide with, the relevant intersection. Equality holds precisely when every one of these geometric defects vanishes, as stated in the following corollary.

\begin{corollary}[Equality in the Kinser inequality]
\label{cor:Kinser-equality}
Recall that $K=\bigcap_{i=3}^{n}V_i$. Then $\mathcal{K}_n(V_1,\ldots,V_n)=0$ if and only if all of the following conditions hold:
\begin{enumerate}
\item $K\subseteq V_1+V_2$,
\item $(V_1+V_3)\cap(V_1+V_n)=V_1+K$,
\item $(K+V_1)\cap(K+V_2)=K$,
\item $(V_i+V_2)\cap(V_i+V_{i-1})=V_i+\bigcap_{j=2}^{i-1}V_j$ for every $4\leq i\leq n$.
\end{enumerate}
\end{corollary}
\begin{proof}
Every term on the right-hand side of \eqref{Kinser-slack-decomposition} is nonnegative. Hence their sum vanishes if and only if each term vanishes. The four stated conditions are precisely the corresponding equality conditions.
\end{proof}

\begin{remark}\label{rem:Ingleton-slack}
For $n=4$, one has $K=V_3\cap V_4$, and Theorem~\ref{thm:Kinser-slack-decomposition} gives an exact decomposition of the Ingleton slack. In particular, the terms
\begin{equation*}
[(V_1+V_3)\cap(V_1+V_4):V_1+(V_3\cap V_4)]
\end{equation*}
and
\begin{equation*}
[(V_4+V_2)\cap(V_4+V_3):V_4+(V_2\cap V_3)]
\end{equation*}
are of the same three-subspace non-distributivity form that appears in the discoordination invariant of Friedman and Tootooni \cite{FT23}. Thus \eqref{Kinser-slack-decomposition} may be viewed as a higher-order resolution of the geometric defects occurring in Ingleton--Kinser inequalities.
\end{remark}

We next connect the Kinser hierarchy with the successive defect gaps $B_k=A_{k+1}-A_k$ introduced in \eqref{defect-gap-Bk}. Although the individual Kinser inequalities contain representability-specific information, their complete symmetrization has a particularly simple form.

For $\pi\in\mathfrak{S}_n$, write
\begin{equation}
\mathcal{K}_n^{\pi}\coloneq\mathcal{K}_n\bigl(V_{\pi(1)},\ldots,V_{\pi(n)}\bigr). 
\end{equation}

\begin{theorem}[Symmetrized Kinser slack]
\label{thm:symmetrized-Kinser}
For every arrangement $V_1,\ldots,V_n\subseteq V$, with $n\ge4$,
\begin{equation}
\frac{1}{n!}\sum_{\pi\in\mathfrak{S}_n}\mathcal{K}_n^{\pi}
=(n-2)B_2.
\end{equation}
Equivalently, in terms of the average rank profile,
\begin{equation}
\frac{1}{n!}\sum_{\pi\in\mathfrak{S}_n}\mathcal{K}_n^{\pi}
=(n-2)(2R_2-R_1-R_3).
\end{equation}
\end{theorem}
\begin{proof}
Under uniform permutation averaging, a rank term on $k$ distinct labels has average $R_k$. In \eqref{Kinser-slack}, the total coefficients of singleton, pair, and triple ranks are respectively $-(n-2)$, $3+2(n-3)-1=2(n-2)$, and $-(n-2)$. The average is therefore $(n-2)(2R_2-R_1-R_3)$, which equals $(n-2)B_2$ by \eqref{defect-gap-Bk-Rk}.
\end{proof}

\begin{remark}\label{rem:Kinser-symmetry-collapse}
Theorem~\ref{thm:symmetrized-Kinser} shows that complete permutation symmetry removes the additional directional information carried by an individual Kinser inequality: the fully averaged Kinser slack is exactly a multiple of the first curvature $B_2$ of the average rank profile. Thus the representability-specific information in the Kinser hierarchy is inherently nonsymmetric with respect to the labels, whereas its fully symmetric component is already encoded by the graded defect hierarchy.
\end{remark}

The preceding identity extends by contraction to recover the higher curvatures $B_k$. Fixing $S$ and passing to the quotient by its span treats those subspaces as already present in the background. For $S\subseteq[n]$, consider the quotient arrangement $(\overline{V}_j^{\,S})_{j\in[n]\setminus S}$ defined in
\eqref{quotient-subspaces}. If $E$ is a finite set, write $\mathfrak{S}(E)$ for the set of its orderings, identified with bijections $[|E|]\to E$. When $n-|S|\geq4$, for $\pi\in\mathfrak{S}([n]\setminus S)$, let $\mathcal{K}_{n-|S|}^{\,S,\pi}$ denote the Kinser slack of the ordered quotient arrangement $(\overline{V}_{\pi(1)}^{\,S},\ldots,\overline{V}_{\pi(n-|S|)}^{\,S})$. Averaging these slacks over the choices of the contracted set $S$ of size $k-2$ and the orderings of the remaining subspaces yields $B_k$.

\begin{theorem}[Conditioned Kinser representation of the defect gaps]\label{thm:conditioned-Kinser}
Let $n\ge4$. For every $2\leq k\leq n-2$,
\begin{equation}\label{defect-gap-Bk-Kinser-average}
B_k=\frac{1}{(n-k)\binom{n}{k-2}} \sum_{\substack{S\subseteq[n]\\|S|=k-2}}
\frac{1}{(n-k+2)!}\sum_{\pi\in\mathfrak{S}([n]\setminus S)}
\mathcal{K}_{n-k+2}^{S,\pi}.
\end{equation}
\end{theorem}
\begin{proof}
Fix $S\subseteq[n]$ with $|S|=k-2$, and set $m\coloneq n-k+2\ge4$. Let $(R_j^{S})_{j=0}^{m}$ be the average rank profile of the quotient arrangement. Explicitly, for $0\leq j\leq m$,
\begin{equation*}
R_j^{S}=\frac{1}{\binom{m}{j}}
\sum_{\substack{T\subseteq[n]\setminus S\\|T|=j}}
\Bigl[\dim\Bigl(U_S+\sum_{i\in T}V_i\Bigr)-\dim(U_S)\Bigr].
\end{equation*}
Applying Theorem~\ref{thm:symmetrized-Kinser} to the $m$
quotient subspaces gives
\begin{equation}\label{fixed-S-Kinser}
\frac{1}{m!}\sum_{\pi\in\mathfrak{S}([n]\setminus S)}
\mathcal{K}_{m}^{S,\pi}=(n-k)(2R_2^{S}-R_1^{S}-R_3^{S}).
\end{equation}
For $0\leq j\leq m$, averaging over all $(k-2)$-element subsets $S$ and double counting disjoint pairs $(S,T)$ with $|T|=j$ yields
\begin{equation}\label{quotient-rank-average}
\frac{1}{\binom{n}{k-2}}
\sum_{\substack{S\subseteq[n]\\|S|=k-2}}R_j^{S}
=R_{k-2+j}-R_{k-2}.
\end{equation}
Indeed, every $(k-2+j)$-element set occurs as $S\cup T$
exactly $\binom{k-2+j}{k-2}$ times. Consequently, averaging \eqref{fixed-S-Kinser} and using \eqref{quotient-rank-average} gives
\begin{equation*}
\frac{1}{\binom{n}{k-2}}
\sum_{\substack{S\subseteq[n]\\|S|=k-2}}
\frac{1}{m!}\sum_{\pi\in\mathfrak{S}([n]\setminus S)}
\mathcal{K}_{m}^{S,\pi}=(n-k)(2R_k-R_{k-1}-R_{k+1})
=(n-k)B_k,
\end{equation*}
where the last equality follows from \eqref{defect-gap-Bk-Rk}. Dividing by $n-k$ proves \eqref{defect-gap-Bk-Kinser-average}.
\end{proof}

In particular, for $2\leq k\leq n-2$, $B_k$ is $1/(n-k)$ times the average Kinser slack after contracting $k-2$ subspaces. Thus the curvature has two readings: it is the average drop in marginal rank, and, for representable arrangements, it is also an averaged slack of a representability-specific inequality on quotient arrangements.

The numerical averaging identity also holds for arbitrary set functions, using contraction $\mathfrak{r}_S(T)=\mathfrak{r}(S\cup T)-\mathfrak{r}(S)$. Linear representability supplies the nonnegative individual Kinser slacks and therefore the following rigidity criterion for vanishing curvature.

\begin{corollary}[Kinser rigidity for vanishing defect curvature]\label{cor:Kinser-Br-zero}
For $2\leq k\leq n-2$, one has $B_k=0$ if and only if $\mathcal{K}_{n-k+2}^{S,\pi}=0$ for every $S\subseteq[n]$ with $|S|=k-2$ and every $\pi\in\mathfrak{S}([n]\setminus S)$. Equivalently, every ordered quotient arrangement obtained by contracting a set of $k-2$ subspaces attains equality in the corresponding Kinser inequality.
\end{corollary}
\begin{proof}
Every quotient arrangement in \eqref{defect-gap-Bk-Kinser-average} is again linearly representable. Hence each Kinser slack $\mathcal{K}_{n-k+2}^{S,\pi}$ is nonnegative by \eqref{Kinser-nonnegative}. Formula~\eqref{defect-gap-Bk-Kinser-average} therefore expresses $B_k$ as a positive constant times the average of nonnegative quantities. The average vanishes if and only if every summand vanishes.
\end{proof}

There can be no cancellation in this conclusion since every contracted Kinser slack is nonnegative. Consequently, $B_k=0$ forces every contracted ordering to attain equality individually, not just on average.

\begin{remark}\label{rem:Kinser-defect-bridge}
Theorem~\ref{thm:conditioned-Kinser} provides a direct bridge between two different hierarchies. The quantities $B_k$ arise from the symmetric dimension-defect hierarchy and measure the discrete curvature of the average rank profile, while the Kinser inequalities encode additional constraints imposed by linear representability. For $2\leq k\leq n-2$, the curvature $B_k$ equals $1/(n-k)$ times the average Kinser slack of quotient arrangements obtained by contracting $k-2$ subspaces, where the average is taken over both the contraction sets and the orderings of the remaining subspaces. In particular, the full symmetrization of a representability-specific inequality does not produce a new symmetric inequality beyond the submodular curvature $B_k\geq0$; instead, the additional representability information is retained in the individual, nonsymmetrized Kinser slacks. The final curvature $B_{n-1}$ has no counterpart in this construction, since the Kinser hierarchy begins with four subspaces.
\end{remark}

\begin{example}\label{ex:final-curvature}
%[Dependence confined to the final curvature]
Let $n\geq4$, let $e_1,\ldots,e_{n-1}$ be the standard basis of $\mathbbm{F}^{n-1}$, and set
\[
V_i=\operatorname{span}\{e_i\}\quad(1\leq i\leq n-1),
\qquad\qquad
V_n=\operatorname{span}\{e_1+\cdots+e_{n-1}\}.
\]
Every proper selection of these lines is independent, whereas the full selection has rank $n-1$. Thus
\[
R_k=\min\{k,n-1\},\qquad A_2=\cdots=A_{n-1}=0,\qquad A_n=1,
\]
and
\[
B_2=\cdots=B_{n-2}=0,\qquad B_{n-1}=1.
\]
By Corollary~\ref{cor:Kinser-Br-zero}, every contracted Kinser slack appearing in Theorem~\ref{thm:conditioned-Kinser} vanishes, although the total dimension defect is $1$. This arrangement exhibits explicitly the dependence measured by the final curvature and omitted from that contraction construction.
\end{example}

The results of this subsection do not address the completeness problem for linear-rank inequalities posed by Kinser. Rather, they describe the geometric slack and the symmetric component of the Kinser hierarchy. In particular, complete permutation averaging collapses the representability-specific Kinser slack to the submodular curvature $B_2$, so these particular fully averaged inequalities add no constraint beyond submodularity.

%%%%%%%%%%%%%%%%%%%%%%%%%%%%%%%%%%%%%%%%%%%%%%%%%
\section{Discussion and outlook}
%%%%%%%%%%%%%%%%%%%%%%%%%%%%%%%%%%%%%%%%%%%%%%%%%
The symmetric Grassmann formula developed here expresses the dimension of a finite sum through intersections of individual subspaces with partial sums. Its explicit coefficients depend only on the number of subspaces involved and arise by averaging the ordered Grassmann identity. This provides a nonrecursive organization of dimension data complementary to the recursive formulas of Tian and Makrides--Szemberg \cite{Tian19,MS25,Tian25}. The connection with the Shapley value explains both the coefficients and the interpretation of each subspace's contribution as its expected marginal dimension.

The graded profile $A_2,\ldots,A_n$ turns this formula into a description of how overlap develops as more subspaces are included. Its successive gaps yield exact endpoint remainders, coincide with the discrete curvature of the average rank profile, and characterize endpoint equality through a direct sum modulo the common intersection. The monotonicity has a submodular origin, while the equality theorem records the corresponding subspace geometry. The annihilator construction gives parallel statements for intersections.

The symmetric defect quantities introduced above are further related to representability-specific linear-rank inequalities. Retaining the nonnegative remainders in Kinser's proof yields explicit quotient dimensions and corresponding equality conditions. Their full permutation average gives $(n-2)B_2$, while contraction recovers $B_k$ for $2\leq k\leq n-2$. Thus ordered Kinser slacks and averaged defect curvature describe related aspects of the same subspace arrangement. The averaged identities do not produce new independent linear-rank inequalities; the additional representability-specific information remains in the ordered slacks. The final curvature $B_{n-1}$ lies outside this contraction construction.

Several directions remain natural. One is to obtain a structural classification, up to linear isomorphism, of the arrangements satisfying $B_k=0$ at a fixed intermediate level, using the quotient-intersection criterion in Corollary~\ref{cor:Bk-quotient}. Another is to study weighted or partially symmetric orderings that preserve more information about a structured family. Analogous formulas in ranked modular lattices, and the behavior of the defect profile under tensor products or random subspace models, also merit investigation. Since representable polymatroids occur in network coding and information inequalities \cite{CG08,DFZ07}, it is natural to ask whether the explicit intersection losses and equality criteria can sharpen dimension arguments in those settings. Such applications require a separate analysis of the underlying subspace arrangements.

%%%%%%%%%%%%%%%%%%%%%%%%%%%%%%%%%%%%%%%%%%%%%%%%%
\section*{Acknowledgments}
M.~G. acknowledges helpful discussions with Shmuel Friedland on inclusion--exclusion for subspaces, particularly his observations on the role of exact sequences and Grothendieck groups. This work was supported by the European Union through the ERC Advanced Grant \emph{TAtypic}, Project No.~101142236. Views and opinions expressed are, however, those of the authors only and do not necessarily reflect those of the European Union or the European Research Council Executive Agency. Neither the European Union nor the granting authority can be held responsible for them.

%%%%%%%%%%%%%%%%%%%%%%%%%%%%%%%%%%%%%%%%%%%%%%%%%
\appendix
%%%%%%%%%%%%%%%%%%%%%%%%%%%%%%%%%%%%%%%%%%%%%%%%%
\section{An inductive proof of the symmetric Grassmann formula}
\label{app:inductive-proof}
%%%%%%%%%%%%%%%%%%%%%%%%%%%%%%%%%%%%%%%%%%%%%%%%%
We give a second proof of Theorem~\ref{thm:symmetric-Grassmann} by averaging the two-subspace Grassmann formula over the choice of the last subspace. We prove by induction on $n$ that
\begin{equation*}
\dim\Bigl(\sum_{i\in[n]}V_i\Bigr)=\sum_{i\in[n]}
\Bigl[\dim(V_i)-\sum_{\emptyset\neq S\subseteq[n]\setminus\{i\}}c_{n,|S|+1}\,\dim\Bigl(V_i\cap\sum_{j\in S}V_j\Bigr)\Bigr].
\end{equation*}

\begin{proof}[Alternative proof of Theorem~\ref{thm:symmetric-Grassmann}]
We proceed by induction on $n$. For $n=2$, each of the two pointed terms $\dim(V_1\cap V_2)$ and $\dim(V_2\cap V_1)$ has coefficient $c_{2,2}=1/2$. Together they give the classical Grassmann formula. Assume the formula holds for $n-1$ subspaces. For each $i\in[n]$, set
\begin{equation*}
W_{i}\coloneq\sum_{j\in[n]\setminus\{i\}}V_j.
\end{equation*}
Then $\sum_{j\in[n]}V_j=W_{i}+V_i$, so the two-subspace Grassmann formula gives
\begin{equation*}
\dim\Bigl(\sum_{j\in[n]}V_j\Bigr)=\dim(W_{i})+\dim(V_i)-\dim(W_{i}\cap V_i).
\end{equation*}
Since this identity holds for every $i\in[n]$, averaging over all $i$ yields
\begin{equation*}
\dim\Bigl(\sum_{j\in[n]}V_j\Bigr)=\frac1n\sum_{i\in[n]}
\Bigl[\dim(W_{i})+\dim(V_i)-\dim(W_{i}\cap V_i)\Bigr].
\end{equation*}
By the induction hypothesis, for each fixed $i$,
\begin{equation*}
\dim(W_{i})=\sum_{j\in[n]\setminus\{i\}}
\Bigl[\dim(V_j)-\sum_{\emptyset\neq S\subseteq[n]\setminus\{i,j\}}c_{n-1,|S|+1}\,\dim\Bigl(V_j\cap\sum_{l\in S}V_l\Bigr)\Bigr].
\end{equation*}
Substitute these expansions into the averaged identity. We determine the coefficient of each dimension term and each intersection term separately. The terms involving only individual dimensions simplify to
\begin{equation*}
\frac{1}{n}\sum_{i\in[n]}
\Bigl(\sum_{j\in[n]\setminus\{i\}} \dim(V_j)+\dim(V_i)\Bigr)
=\frac1n\sum_{i\in[n]}\sum_{j\in[n]}\dim(V_j)
=\sum_{j\in[n]}\dim(V_j).
\end{equation*}
Now fix $j\in[n]$ and a nonempty proper subset $S\subsetneq[n]\setminus\{j\}$. Put $k\coloneq|S|+1$; then $2\leq k<n$. We count the choices of the omitted index $i$ for which this fixed correction term appears. The term $\dim(V_j\cap\sum_{l\in S}V_l)$ occurs in the inductive expansion of $\dim(W_{i})$ exactly when $i\notin S\cup\{j\}$. There are $n-k$ such choices of $i$. Hence its coefficient in the averaged correction is
\begin{equation*}
\frac{n-k}{n}\,c_{n-1,k}=\frac{n-k}{n}\,\frac{1}{k\binom{n-1}{k}}=\frac{1}{k\binom{n}{k}}=c_{n,k}.
\end{equation*}
Finally, when $S=[n]\setminus\{j\}$, the term $\dim(W_{j}\cap V_j)$ occurs directly in the Grassmann formula for the single choice $i=j$. Its coefficient is $1/n=c_{n,n}$. Thus every correction term has the coefficient asserted in \eqref{sym-Grassmann}. %Reindexing by $T=S\cup\{j\}$ also gives \eqref{sym-Grassmann-r}.
\end{proof}

%%%%%%%%%%%%%%%%%%%%%%%%%%%%%%%%%%%%%%%%%%%%%%%%%
\section{An information-theoretic interpretation}
\label{app:information-theoretic}
%%%%%%%%%%%%%%%%%%%%%%%%%%%%%%%%%%%%%%%%%%%%%%%%%
Throughout this appendix, $\log$ denotes the base-$2$ logarithm. For a finite-valued random variable $X$ with probability mass function $\mathtt{p}_X$, define its entropy by
\begin{equation}
H(X)\coloneq-\sum_x\mathtt{p}_X(x)\log\mathtt{p}_X(x).
\end{equation}
For finite-valued random variables $X,Y,Z$, define the mutual information and conditional mutual information by
\begin{align}
I(X;Y)&\coloneq H(X)+H(Y)-H(X,Y), \\
I(X;Y\mid Z)&\coloneq H(X,Z)+H(Y,Z)-H(Z)-H(X,Y,Z),
\end{align}
where joint arguments denote joint entropy. See, for example, \cite[Chapter~2]{CT06}.

We first relate these quantities to the subspaces in the paper. Assume that $V$ is over a finite field $\mathbbm{F}_q$ with $q$ elements, and put $d\coloneq\dim(V)$. After choosing a basis, identify $V$ with $\mathbbm{F}_q^d$. Let $Z$ be a uniformly random vector in $\mathbbm{F}_q^d$, so each vector has probability $q^{-d}$. For each subspace $V_i\subseteq\mathbbm{F}_q^d$, choose a matrix $G_i$ whose row space is $V_i$, and set $X_i\coloneq G_iZ$. For $S\subseteq[n]$, write $X_S=(X_i)_{i\in S}$, with $H(X_\emptyset)=0$. The linear map $Z\mapsto X_S$ has rank $\dim\bigl(\sum_{i\in S}V_i\bigr)$, and therefore
\begin{equation}\label{linear-source-entropy}
H(X_S)=\dim\Bigl(\sum_{i\in S}V_i\Bigr)\log q=\mathfrak{r}(S)\log q.
\end{equation}
Thus, for this linear source, the entropy of $X_S$ is the dimension of the span of the corresponding subspaces, multiplied by $\log q$.

For distinct $i,j\notin S$, the conditional mutual information has the geometric form
\begin{equation}\label{linear-source-cmi}
I(X_i;X_j\mid X_S)=
\Bigl[\dim\Bigl(V_i\cap\Bigl(V_j+\sum_{l\in S}V_l\Bigr)\Bigr)-\dim\Bigl(V_i\cap\sum_{l\in S}V_l\Bigr)
\Bigr]\log q.
\end{equation}

We now state the corresponding identities for arbitrary finite-valued random variables $X_1,\ldots,X_n$. For $S\subseteq[n]$, write $X_S=(X_i)_{i\in S}$, with $H(X_\emptyset)=0$. For $0\leq k\leq n$, define the average entropy profile by
\begin{equation}\label{average-entropy}
\overline{H}_k\coloneq\frac{1}{\binom{n}{k}} \sum_{\substack{S\subseteq[n]\\|S|=k}}H(X_S),
\qquad \overline{H}_0=0.
\end{equation}
For $2\leq k\leq n$ and $2\leq k\leq n-1$, respectively, define
\begin{equation}\label{entropy-defects}
A_k^{H}\coloneq\overline{H}_1+\overline{H}_{k-1}-\overline{H}_k, \qquad B_k^{H}\coloneq A_{k+1}^{H}-A_k^{H}.
\end{equation}

\begin{proposition}\label{prop:entropy-correspondence}
For $2\leq k\leq n$, the average mutual information is $A_k^{H}$:
\begin{equation}\label{Ak-mutual-information}
A_k^{H}=c_{n,k}\sum_{\substack{T\subseteq[n]\\|T|=k}} \sum_{i\in T} I(X_i;X_{T\setminus\{i\}}).
\end{equation}
Here $c_{n,k}\coloneq(k\binom{n}{k})^{-1}$. For $2\leq k\leq n-1$, the average conditional mutual information is $B_k^{H}$:
\begin{equation}\label{Bk-conditional-mutual-information}
B_k^{H}=\frac{c_{n,k}}{n-k}\sum_{i\in[n]} \sum_{\substack{S\subseteq[n]\setminus\{i\}\\|S|=k-1}} \sum_{j\in[n]\setminus(S\cup\{i\})}I(X_i;X_j\mid X_S).
\end{equation}
Moreover, the total correlation satisfies
\begin{equation}\label{total-correlation}
\sum_{i\in[n]}H(X_i)-H(X_1,\ldots,X_n)=\sum_{k=2}^{n}A_k^{H}.
\end{equation}
In particular, $A_k^{H}\geq0$ and $B_k^{H}\geq0$, and the average entropy profile $(\overline{H}_k)_{k=0}^n$ is discretely concave.
\end{proposition}
\begin{proof}
For each $T\subseteq[n]$ and $i\in T$,
\[
I(X_i;X_{T\setminus\{i\}})=H(X_i)+H(X_{T\setminus\{i\}})-H(X_T).
\]
Averaging over all pairs $(T,i)$ with $|T|=k$ gives \eqref{Ak-mutual-information}, since the three terms have averages $\overline{H}_1$, $\overline{H}_{k-1}$, and $\overline{H}_k$, respectively. Taking successive differences in \eqref{entropy-defects} gives
\[
B_k^{H}=2\overline{H}_k-\overline{H}_{k-1}-\overline{H}_{k+1}.
\]
For distinct $i,j\notin S$,
\[
I(X_i;X_j\mid X_S)=H(X_{S\cup\{i\}})+H(X_{S\cup\{j\}})-H(X_S)-H(X_{S\cup\{i,j\}}).
\]
Average this identity over all triples $(S,i,j)$ with $|S|=k-1$ and distinct $i,j\notin S$. The four terms have averages $\overline{H}_k$, $\overline{H}_k$, $\overline{H}_{k-1}$, and $\overline{H}_{k+1}$, respectively. There are $n\binom{n-1}{k-1}(n-k)$ such triples. This gives \eqref{Bk-conditional-mutual-information}. Finally, summing $A_k^{H}=\overline{H}_1+\overline{H}_{k-1}-\overline{H}_k$ over $2\leq k\leq n$ gives \eqref{total-correlation}. Nonnegativity of mutual and conditional mutual information gives $A_k^H\geq0$ and $B_k^H\geq0$, respectively. For $2\leq k\leq n-1$, $2\overline{H}_k-\overline{H}_{k-1}-\overline{H}_{k+1}=B_k^H\geq0$, while the first step is $2\overline{H}_1-\overline{H}_0-\overline{H}_2=A_2^H\geq0$. Together these inequalities establish discrete concavity of the full average entropy profile.
\end{proof}

For the linear source constructed above, \eqref{linear-source-entropy} implies that its average entropy profile satisfies $\overline{H}_k=R_k\log q$. Consequently, its entropy defects satisfy
\begin{equation}\label{linear-source-defects}
A_k^{H}=A_k\log q, \qquad B_k^{H}=B_k\log q.
\end{equation}
Thus the average mutual information in \eqref{Ak-mutual-information} equals $A_k\log q$, and the average conditional mutual information in \eqref{Bk-conditional-mutual-information} equals $B_k\log q$. For arbitrary finite-valued random variables, the identities in Proposition~\ref{prop:entropy-correspondence} still hold, although their average entropy profiles need not arise from a linear source.

%%%%%%%%%%%%%%%%%%%%%%%%%%%%%%%%%%%%%%%%%%%%%%%%%

%%%%%%%%%%%%%%%%%%%%%%%%%%%%%%%%%%%%%%%%%%%%%%%%%

\end{document}